\documentclass[a4paper,10pt,reqno]{amsart}
\usepackage{amsfonts,amsmath,amssymb,bbm,amsthm,amsbsy}
\usepackage{graphicx}
\usepackage{geometry}
\usepackage{float}
\usepackage{verbatim}
\usepackage{setspace}
\usepackage{siunitx}
\usepackage{paralist}
\usepackage{enumitem}
\usepackage{xcolor}
\usepackage{lipsum}
\usepackage[english]{babel}
\usepackage{booktabs}
\usepackage{array}
\usepackage{subfig}
\usepackage{caption}
\usepackage{mathtools}
\usepackage{stmaryrd}
\usepackage{stackengine,scalerel}
\usepackage{hyperref}
\usepackage[utf8]{inputenc}

\newenvironment{frcseries}{\fontfamily{frc}\selectfont}{}

\makeatletter
\newtheorem*{rep@theorem}{\rep@title}
\newcommand{\newreptheorem}[2]{
\newenvironment{rep#1}[1]{
 \def\rep@title{#2 \ref{##1}}
 \begin{rep@theorem}}
 {\end{rep@theorem}}}
\makeatother

\newtheorem{thm}{Theorem}[section]
\newtheorem{lemma}[thm]{Lemma}
\newtheorem{prop}[thm]{Proposition}
\newtheorem{corr}[thm]{Corollary}

\newtheorem*{thm*}{Theorem}
\newtheorem*{lemma*}{Lemma}
\newtheorem*{prop*}{Proposition}
\newtheorem*{corr*}{Corollary}
\newtheorem*{claim*}{Claim}

\newtheorem{innercustomthm}{Theorem}
\newenvironment{customthm}[1]
  {\renewcommand\theinnercustomthm{#1}\innercustomthm}
  {\endinnercustomthm}

\theoremstyle{remark}
\newtheorem{rmk}[thm]{Remark}

\newtheorem*{rmk*}{Remark}
\newtheorem*{conj*}{Conjecture}
\newtheorem*{quest*}{Question}

\theoremstyle{definition}
\newtheorem{defn}[thm]{Definition}
\newtheorem{exmp}[thm]{Example}

\newtheorem*{defn*}{Definition}
\newtheorem*{exmp*}{Example}

\newreptheorem{theorem}{Theorem}
\newreptheorem{corollary}{Corollary}
\newreptheorem{proposition}{Proposition}

\usepackage{fancyhdr}
\usepackage{tikz}
\usepackage{tikz-cd}

\usepackage{IEEEtrantools}

\newenvironment{equ*}[1]{\begin{IEEEeqnarray*}{#1}}{\end{IEEEeqnarray*}}

\newcommand{\Z}{\mathbb{Z}}

\DeclareFontFamily{U}{mathx}{}
\DeclareFontShape{U}{mathx}{m}{n}{<-> mathx10}{}
\DeclareSymbolFont{mathx}{U}{mathx}{m}{n}
\DeclareMathAccent{\widecheck}{0}{mathx}{"71}

\newcommand{\inj}{\hookrightarrow}
\newcommand{\sur}{\twoheadrightarrow}

\DeclareMathOperator{\Hom}{Hom}

\DeclareMathOperator{\Ker}{Ker}

\DeclareMathOperator{\colim}{colim}

\newcommand{\Set}{\mathbf{Set}}
\newcommand{\Grp}{\mathbf{Grp}}
\newcommand{\Grpd}{\mathbf{Grpd}}

\newcommand{\Pro}{\mathbf{Pro}}

\newcommand{\Cat}{\mathbf{Cat}}
\newcommand{\Graph}{\mathbf{Graph}}

\newcommand{\G}{\mathcal{G}}
\newcommand{\Hc}{\mathcal{H}}

\title{Categorical Bass--Serre Theory for Tree Quotients}
\author{Jiacheng Tang}
\thanks{Email: jiacheng.tang@postgrad.manchester.ac.uk; Address: Alan Turing Building, University of Manchester, Manchester, M13 9PY, United Kingdom.}

\calclayout
\begin{document}
\maketitle

\begin{abstract}
We develop a categorical framework for Bass--Serre theory, which recovers parts of classical and profinite Bass--Serre theories for tree quotients. More precisely, we show that in what we call a \emph{Bass--Serre category}, a group object acting on a graph satisfying some assumptions can be recovered as the fundamental group of the associated graph of groups acting on the standard graph. Examples of Bass--Serre categories include the category of sets, the category of profinite spaces, and any Grothendieck topos.
\end{abstract}

\section{Introduction}
\label{sec1}

Bass--Serre theory, the theory of groups acting on trees, is a major part of geometric group theory with numerous applications to group theory (see \cite{trees}). Since its initial development, the theory has been extended to profinite groups, groupoids, and inverse monoids (see \cite{ribesgraph}, \cite{bsgrpd}, and \cite{bsinv} respectively). As such, one natural question to ask is whether there is a general categorical framework that captures the above theories.

The ``structure theorem" of classical Bass--Serre theory (\cite[Theorem I.13]{trees}) asserts that if a group $G$ acts on a tree $X$, then $G$ (resp.\ $X$) can be recovered as the fundamental group (resp.\ the standard graph) of the associated graph of groups. There are, in a vague sense, two parts to the theory: the ``geometric" or ``homotopy-theoretic" part of showing that the standard graph is the universal covering of $X$, and the ``algebraic" part of reconstructing $G$ as the fundamental group.

\subsection*{Aim and Main Results} 

The current paper aims to build a categorical framework for the ``algebraic" part of Bass--Serre theory when the quotient $G\backslash X$ is a tree. We define what it means for a category $\mathbf{C}$ to be \emph{Bass--Serre} (Definition \ref{bscat}), which roughly asks for the following two properties. Firstly, whenever a group object $G$ acts on an object $X\in\mathbf{C}$, we want the quotient $G\backslash X$ to be well behaved. We call this property \emph{action Barr-exactness} (Definition \ref{defweakexact}), which is a weakening of the usual notion of Barr-exactness\footnote{This weakening is very much necessary for us, as the category of profinite spaces is action Barr-exact but not Barr-exact (see Example \ref{eglocallyweakly}).}. This in particular implies that $\mathbf{C}$ satisfies an internal analogue of the Orbit-Stabiliser Theorem (see Definition \ref{defost} and Proposition \ref{propexactost}).

Secondly, to even define the fundamental group of a graph of groups (Definition \ref{fundgrp}), we need the category $\Grp(\mathbf{C})$ of group objects in $\mathbf{C}$ to have $\mathbf{C}$-internal coproducts (see Section \ref{sec3-2}). When $\mathbf{C}=\Set$ is the category of sets, this is just the coproduct of groups. When $\mathbf{C}=\Pro$ is the category of profinite spaces, this is the profinite coproduct, or ``free profinite product", of profinite groups indexed by a profinite space.

Suppose $\mathbf{C}$ is a Bass--Serre category, which includes $\Set$, $\Pro$, and every Grothendieck topos (Propositions \ref{bssetpro} and \ref{bssheaf}). If a group object $G\in\Grp(\mathbf{C})$ acts on an internal graph $X\in\Graph(\mathbf{C})$ satisfying some assumptions\footnote{As we shall point out in Remark \ref{rmktransport}, for $\mathbf{C}=\Set$, we are implicitly assuming that the quotient $G\backslash X$ is a tree.}, then we can form the associated graph of groups (Section \ref{sec3-1}), its fundamental group $\pi$ (Definition \ref{fundgrp}), and the standard graph $T$ (Definition \ref{stdgraph}), along with canonical maps $\varphi\colon\pi\to G$ and $\Phi\colon T\to X$. Our main theorem, which is a categorical version of \cite[Theorem I.10]{trees} (often called the ``structure theorem"), is as follows.

\begin{customthm}{\ref{mainthmre}}
With the assumptions above, the following are equivalent.
\begin{enumerate}[label=(\roman*)]
\item\label{mainthmintro1} The map $\Phi\colon T\to X$ has a section in $\Graph(\mathbf{C})$.
\item The map $\Phi\colon T\to X$ is an isomorphism in $\Graph(\mathbf{C})$.
\item The map $\varphi\colon\pi\to G$ is an isomorphism in $\mathbf{C}$ (and hence in $\Grp(\mathbf{C})$).
\end{enumerate}
\end{customthm}

As we will see in Example \ref{mainthmeg}, for $\mathbf{C}=\Set$ (resp.\ $\mathbf{C}=\Pro$) where there is a good covering theory, condition \ref{mainthmintro1} may be viewed as an algebraic way of saying that $X$ is a tree (resp.\ simply connected profinite graph).

In Theorem \ref{mainconv}, we will also establish a partial converse to the structure theorem, namely the fact that any ``monic" graph of groups can be recovered as the graph of groups associated to the action of its fundamental group $\pi$ on the standard graph $T$.

\subsection*{Outline of Paper}

We will assume knowledge of basic category theory (refer to \cite{maccat}), including regular and Barr-exact categories (refer to \cite{barr}). In Section \ref{sec3-2} and the appendices, we will respectively assume knowledge of internal and extensive categories (refer to \cite{handbook} and \cite{extensive}), and of 2-categories (refer to \cite{lack2}), but these sections can be skipped without affecting the core idea of the paper. In various examples, we will also use properties of pro-completions (refer to \cite{catandsh}). Bass--Serre theory for ordinary groups (refer to \cite{trees} or \cite{dicksdun}) and profinite groups (refer to \cite{ribesgraph}) is needed to appreciate (though not needed to understand) the main results in the paper.

In Section \ref{sec2}, we will study some basic concepts related to internal group actions, including stabilisers and quotients (Section \ref{sec2-1}) and an internal analogue of the Orbit-Stabiliser Theorem (Definition \ref{defost}). We will give, in Section \ref{sec2-3}, an appropriate weakening of Barr-exactness that includes $\Pro$ and implies the internal Orbit-Stabiliser Theorem. The focus of Section \ref{sec3} is then to define graphs of groups and their fundamental groups using internal colimits. We will also prove the ``surjectivity" part of the structure theorem in Theorem \ref{thmregepic}, namely that if $X$ is connected, then the canonical map $\varphi\colon\pi\to G$ is regular epic.

In Section \ref{sec4}, we define the standard graph of a graph of groups and prove the structure theorem (Theorem \ref{mainthm}), as well as its converse (Proposition \ref{propgraphconv}). The proof of the structure theorem requires a number of technical, seemingly disconnected assumptions, which are satisfied in Bass--Serre categories, whose definition (Definition \ref{bscat}) is much cleaner and easier to verify. Indeed, in Section \ref{sec5}, we will restate our main results for Bass--Serre categories (Theorems \ref{mainthmre} and \ref{mainconv}) and look at some examples of such categories, which include all Grothendieck toposes.

One immediate question is whether one could extend our theory and eliminate the implicit assumption that $G\backslash X$ is a tree (see Remark \ref{rmktransport}). In Appendix \ref{appendixA}, we give some ideas on how this might be done and explain why it is likely to be more technical. In Appendix \ref{appendixB}, we will explain the relationship between our framework and the one developed in \cite{gogpullbacks}.

Remark: We will often write ``=" to mean canonically isomorphic (or canonically equivalent for categories), but the meaning should be clear from context.

Convention: Although Serre's \emph{Trees} (\cite{trees}) is the classic reference for Bass--Serre theory, we will follow the treatment of \cite{dicksdun} instead and work with directed graphs. All our categories are locally small and all (co)limits are small. By 2-categories or 2-(co)limits, which only appear in the appendices, we always mean strict ones, unless otherwise specified.

\subsection*{Acknowledgements}
The author would like to thank Giacomo Tendas for helpful discussions and for reading a draft of this paper.

This work was supported by the EPSRC Doctoral Training Partnership [grant number \linebreak EP/W524347/1].

\subsection*{Statement of AI Use}
ChatGPT 5.6 Sol has been used for idea generation, literature search, and proofreading. In particular, both appendices are inspired by observations made by ChatGPT 5.6 Sol (and earlier versions). The entire manuscript is written and checked by the author.

\section{Internal Group Actions}
\label{sec2}

In this section, we will establish some basic results about internal group actions. After defining stabilisers and quotients in a general category in Section \ref{sec2-1}, we will look at an internal analogue of the Orbit-Stabiliser Theorem (Section \ref{sec2-2}), which we abbreviate to $\mathrm{OST}$. In Section \ref{sec2-3}, we will study the notion of action Barr-exactness, which is sufficient to imply $\mathrm{OST}$ and will form a key component in the definition of Bass--Serre categories in Section \ref{sec5}.

\subsection{Stabilisers and Quotients}
\label{sec2-1}

This subsection aims to define and study basic notions regarding internal groups and graphs, such as stabilisers, free actions, quotients, transitive actions and connected graphs. Let $\mathbf{C}$ be a category with finite limits.

We shall set up some notations first. Let $\Grp(\mathbf{C})$ denote the category of group objects in $\mathbf{C}$ and let $\Graph(\mathbf{C})=[(\boldsymbol{\cdot}\rightrightarrows\boldsymbol{\cdot}),\mathbf{C}]$ denote the category of graphs in $\mathbf{C}$. We will usually write $$m_G\colon G\times G\to G,\,\,\,\, i_G\colon G\to G,\,\,\,\, e_G\colon*\to G$$ for the structure maps of a group object $G\in\Grp(\mathbf{C})$, and write $$X=(EX\overset{d_0}{\underset{d_1}{\rightrightarrows}}VX)$$ for a graph $X$. As for usual (directed) graphs, we call $EX$, $VX$, $d_0$ and $d_1$ the \emph{edge object}, the \emph{vertex object}, the \emph{source map} and the \emph{target map} of $X$, respectively.

Given a group object $G\in\Grp(\mathbf{C})$, let $\mathbf{C}(G)$ denote the category of (left) $G$-objects in $\mathbf{C}$. Here, a $G$-object is just an object $X\in\mathbf{C}$, together with a map $a\colon G\times X\to X$, called the \emph{action map}, satisfying commutative diagrams corresponding to the usual axioms of group actions. We will often just say that \emph{$G$ acts on $X$}. Similarly, let $\Graph(G)=[(\boldsymbol{\cdot}\rightrightarrows\boldsymbol{\cdot}),\mathbf{C}(G)]$ denote the category of $G$-graphs in $\mathbf{C}$. By commutation of limits, we see that $\mathbf{C}(G)$ and hence $\Graph(G)$ have finite limits which are computed as in $\mathbf{C}$.

\begin{defn}\label{stabdef}
Let $G\in\Grp(\mathbf{C})$ and $X$ be a $G$-object (resp.\ $G$-graph). The \emph{stabiliser} of this action, denoted by $G_X\in\mathbf{C}_{/X}$ (resp.\ $\Graph(\mathbf{C})_{/X}$), is the following equaliser in $\mathbf{C}$ (resp.\ $\Graph(\mathbf{C})$):
$$G_{X}\xrightarrow{\mathrm{eq}}G\times X\overset{a}{\underset{\mathrm{pr}_X}{\rightrightarrows}} X.$$ Note that for $\Graph(\mathbf{C})$, by $G$ we really mean $(G\rightrightarrows G)$ (the diagonal functor with value $G$), $X=(EX\rightrightarrows VX)$ and the product is taken in the functor category $\Graph(\mathbf{C})=[(\boldsymbol{\cdot}\rightrightarrows\boldsymbol{\cdot}),\mathbf{C}]$. In other words, we have the following diagram in $\mathbf{C}$:
\[
\begin{tikzcd}
E(G_X) \arrow[r, "\mathrm{eq}"] \arrow[d,shift left]\arrow[d,shift right] & G\times EX \arrow[r,shift left]\arrow[r,shift right] \arrow[d,shift left]\arrow[d,shift right] & EX\arrow[d,shift left]\arrow[d,shift right] \\
V(G_X) \arrow[r, "\mathrm{eq}"']          & G\times VX \arrow[r,shift left]\arrow[r,shift right]           & VX          
\end{tikzcd}
\]
We will be using similar abuse of notation in the following for convenience.
\end{defn}

\begin{exmp}
Suppose $\mathbf{C}=\Set$ and $X$ is a $G$-set/graph. Then, the stabiliser object $G_X$ defined above is precisely the set $\{(g,x)\colon x\in X, g\in G_x\}$, where $G_x$ is the usual stabiliser group of $x\in X$. If $X$ is a $G$-graph, then the source/target map of $G_X$ $$E(G_X)=\{(g,e)\colon e\in EX, g\in G_e\}\to V(G_X)=\{(g,v)\colon v\in VX, g\in G_v\}$$ sends $(g,e)$ to $(g,d_i(e))$. When restricted to the fibre of $E(G_X)$ above $e\in EX$, this is precisely the inclusion of stabilisers $G_e\inj G_{d_i(e)}$. The stabiliser \emph{group} $G_x$ is, of course, a group rather than just a set.
\end{exmp}

\begin{lemma}\label{grpoblemma}
Let $G\in\Grp(\mathbf{C})$ and $X$ be a $G$-object (resp.\ $G$-graph). The stabiliser $G_X\in\mathbf{C}_{/X}$ (resp.\ $\Graph(\mathbf{C})_{/X}$) is naturally a group object in $\mathbf{C}_{/X}$ (resp.\ $\Graph(\mathbf{C})_{/X}$).
\begin{proof}
We treat both cases simultaneously. We need to define multiplication $G_X\times G_X\to G_X$, identity $X\to G_X$ and inversion $G_X\to G_X$ in the relevant slice category over $X$. For example, multiplication is simply given by the following induced map:
\[\begin{tikzcd}
G_X\times_X G_X \arrow[r] \arrow[d, dashed] & G\times G\times X \arrow[d] &   \\
G_X \arrow[r, "\mathrm{eq}"']               & G\times X \arrow[r,shift left]\arrow[r,shift right]         & X
\end{tikzcd}\]
Here, the right vertical map is induced by the multiplication map $m_G\colon G\times G\to G$ of $G$. Similarly, the map $X\to G_X$ is the following induced map:
\[
\begin{tikzcd}
X \arrow[r,equal] \arrow[d, dashed] & *\times X \arrow[d]                                    &   \\
G_X \arrow[r, "\mathrm{eq}"'] & G\times X \arrow[r, shift left] \arrow[r, shift right] & X
\end{tikzcd}
\]
We leave the remainder of the proof as an exercise.
\end{proof}
\end{lemma}

\begin{defn}
Let $G\in\Grp(\mathbf{C})$ and $X\in\mathbf{C}(G)$ (or $X\in\Graph(G)$). We say that the action is \emph{free} if $G_X=X$, or more precisely, if the canonical map $G_X\to X$ is an isomorphism.
\end{defn}

\begin{lemma}\label{selffree}
The canonical action of $G$ on itself is free.
\begin{proof}
This can be checked directly, which is left as an exercise in the axioms of group objects, but instead we will use the Yoneda embedding $y\colon\mathbf{C}\inj[\mathbf{C}^\mathrm{op},\Set]_{\mathrm{small}}$, where $[\mathbf{C}^\mathrm{op},\Set]_{\mathrm{small}}$ is the category of small presheaves. Since $y$ preserves finite limits, it induces a functor on group objects $\Grp(\mathbf{C})\inj[\mathbf{C}^\mathrm{op},\Grp]_{\mathrm{small}}$ which is also denoted by $y$. Here, ``small" refers to presheaves of groups whose underlying presheaf of sets is small. We then only need to show that the composition $yG_G\to yG\times yG\rightrightarrows yG$ is an isomorphism pointwise, but this is obvious.
\end{proof}
\end{lemma}

For the rest of this subsection, suppose that $\mathbf{C}$ also has all coequalisers in addition to having finite limits, and that coequalisers are preserved by the product functor $C\times(-)$ for every $C\in\mathbf{C}$. Many of the results below require only specific coequalisers to exist and be preserved (such as coequalisers coming from group actions), but it is convenient for us to assume that all coequalisers exist from the start.

\begin{defn}\label{transconn}
Let $G\in\Grp(\mathbf{C})$ and $X$ be a $G$-object (resp.\ $G$-graph). The \emph{quotient object} (resp.\ \emph{quotient graph}) of this action, denoted by $G\backslash X\in\mathbf{C}$ (resp.\ $\Graph(\mathbf{C})$), is the following coequaliser in $\mathbf{C}$ (resp.\ $\Graph(\mathbf{C})$): $$G\times X\overset{a}{\underset{\mathrm{pr}_X}{\rightrightarrows}} X\xrightarrow{\mathrm{coeq}}G\backslash X.$$ An action is said to be \emph{transitive} if $G\backslash X=*$ is the terminal object.

Let $X\in\Graph(\mathbf{C})$. Its \emph{object of connected components}, denoted by $\mathrm{Comp}(X)\in\mathbf{C}$, is the following coequaliser in $\mathbf{C}$: $$EX\overset{d_0}{\underset{d_1}{\rightrightarrows}} VX\xrightarrow{\mathrm{coeq}}\mathrm{Comp}(X).$$ Note that $\mathrm{Comp}\colon\Graph(\mathbf{C})\to\mathbf{C}$ is left adjoint to the constant functor $\mathbf{C}\to\Graph(\mathbf{C})$. A graph $X$ is said to be \emph{connected} if $\mathrm{Comp}(X)=*$ is the terminal object.
\end{defn}

\begin{exmp}
\begin{enumerate}[label=(\roman*)]
\item Consider $\mathbf{C}=\Set$. It is easy to see that the notions of freeness, transitivity and connectedness from above all correspond to the usual notions.

\item Consider $\mathbf{C}=\Pro$, the category of profinite spaces. All the notions above are again just the usual ones, but it's slightly less obvious. Let's prove that a profinite graph\footnote{Our notion of profinite graphs is less general than the one in \cite{ribesgraph}, since we require the edge space $EX$ to be a \emph{profinite} subspace of $X$.} $X$ is connected in the sense of Definition \ref{transconn} if and only if it is connected in the sense of \cite{ribesgraph}, i.e.\ it is an inverse limit of finite connected graphs. This follows from the fact that $\Pro$ has exact inverse limits and the inclusion of finite sets into $\Pro$ preserves coequalisers (cf.\ Example \ref{egost}\ref{egost2}), i.e.\ we have in particular $$\mathrm{Comp}(\varprojlim X_i)=\varprojlim\mathrm{Comp}(X_i).$$
\end{enumerate}
\end{exmp}

Let $X\in\Graph(\mathbf{C})$ and consider the diagonal functor $$\Delta\colon\mathbf{C}\to\Graph(\mathbf{C})_{/X},\,\,\,\, C\mapsto (C\times EX\rightrightarrows C\times VX).$$ This is just the composition of the constant functor $\mathbf{C}\to\Graph(\mathbf{C})$ with the product functor $(-)\times X\colon\Graph(\mathbf{C})\to\Graph(\mathbf{C})_{/X}$, each of which has a left adjoint. Thus, the diagonal functor $\Delta$ also has a left adjoint. We record this below along with some other basic results for easy reference later.

\begin{lemma}\label{diaglemma}
The diagonal functor $\Delta$ given above has a left adjoint which sends $Y\to X$ to $\mathrm{Comp}(Y)$.
\end{lemma}

\begin{lemma}\label{deltacoeq}
The diagonal functor $\Delta$ preserves coequalisers.
\begin{proof}
The constant functor $\mathbf{C}\to\Graph(\mathbf{C})$ certainly preserves coequalisers. The product functor $(-)\times X\colon\Graph(\mathbf{C})\to\Graph(\mathbf{C})_{/X}$ does as well by our assumptions on $\mathbf{C}$.
\end{proof}
\end{lemma}

\begin{lemma}\label{deltaff}
The diagonal functor $\Delta$ is fully faithful if and only if $X$ is connected.
\begin{proof}
It is equivalent and easy to verify that the counit $\mathrm{Comp}(\Delta C)\to C$ is an isomorphism if and only if $X$ is connected, using our assumptions on $\mathbf{C}$ again.
\end{proof}
\end{lemma}

\begin{lemma}\label{qconnected}
Let $G\in\Grp(\mathbf{C})$ and $X\in\Graph(G)$. If (the underlying graph of) $X$ is connected, then so is the quotient graph $G\backslash X$.
\begin{proof}
Apply the left adjoint $\mathrm{Comp}$ to the epimorphism $X\to G\backslash X$.
\end{proof}
\end{lemma}

\begin{lemma}\label{selftrans}
The canonical action of $G$ on itself is transitive.
\begin{proof}
This follows from Yoneda as in Lemma \ref{selffree}, but let us also give a direct proof as an illustration.

Given a map $f\colon G\to C$ in $\mathbf{C}$ which coequalises $m_G$ and $\mathrm{pr}$, we would like to construct a dotted map
\[\begin{tikzcd}
G\times G \arrow[r,shift left,"m_G"]\arrow[r,shift right,',"\mathrm{pr}"] & G \arrow[r] \arrow[rd, "f"'] & * \arrow[d, dashed] \\
                    &                              & C                  
\end{tikzcd}\]
making the triangle commute, where $\mathrm{pr}$ denotes the projection onto the second copy of $G$. Once constructed, the dotted map is necessarily unique since $G\to *$ is epic. Define the dotted map to be the composition $$*\xrightarrow{e_G}G\xrightarrow{f}C.$$ Then, the triangle commutes because the two possible compositions below coincide: $$G\xrightarrow{(i_G,\mathrm{id}_G)}G\times G\overset{m_G}{\underset{\mathrm{pr}}{\rightrightarrows}}G\xrightarrow{f}C.$$
\end{proof}
\end{lemma}

\subsection{The Internal Orbit-Stabiliser Theorem}
\label{sec2-2}

In this subsection, we define and study an internal analogue of the Orbit-Stabiliser Theorem, which we abbreviate to $\mathrm{OST}$. Similar in importance to the usual Orbit-Stabiliser Theorem in group theory, $\mathrm{OST}$ will be used in multiple proofs later, such as the surjectivity part of the structure theorem (Theorem \ref{thmregepic}). Throughout this subsection, let $\mathbf{C}$ be a category with finite limits.

Let $G\in\Grp(\mathbf{C})$ and let $X\in\mathbf{C}(G)$ be a $G$-object. Suppose the quotient map $q\colon X\to G\backslash X$ exists and has a fixed section $s\colon G\backslash X\to X$ in $\mathbf{C}$, i.e.\ $q\circ s=\mathrm{id}_{G\backslash X}$. We then have a pullback functor $s^*\colon\mathbf{C}_{/X}\to\mathbf{C}_{/(G\backslash X)}$ which preserves finite limits, and hence group objects. Thus, we obtain from Lemma \ref{grpoblemma} a group object $s^*G_X\in\Grp(\mathbf{C}_{/(G\backslash X)})$.

\begin{defn}\label{defost}
Let $G\in\Grp(\mathbf{C})$, $X\in\mathbf{C}(G)$ and suppose the quotient map $q$ exists and has a fixed section $s$. These data are said to \emph{satisfy $\mathrm{OST}$} (which stands for the ``Orbit-Stabiliser Theorem") if the following diagram in $\mathbf{C}$ is a coequaliser: $$G\times s^*G_X\rightrightarrows G\times (G\backslash X)\xrightarrow{a\circ s} X.$$ Here, the two parallel maps are given by $\mathrm{id}_G\times (s^*G_X\to G\backslash X)$ and by $$G\times s^*G_X\to G\times G_X\to G\times G \times X\xrightarrow{m_G\times q} G\times (G\backslash X).$$ Alternatively, writing $\Delta\colon \mathbf{C}\to\mathbf{C}_{/(G\backslash X)}$ for the diagonal functor, the above definition is equivalent to saying that $\Delta(G)\xrightarrow{a\circ s} X$ is the quotient in $\mathbf{C}_{/(G\backslash X)}$ of $s^*G_X\in\Grp(\mathbf{C}_{/(G\backslash X)})$ acting on $\Delta(G)\in\mathbf{C}_{/(G\backslash X)}$ (on the right) through the canonical map $s^*G_X\inj\Delta(G)$.

We say that the category $\mathbf{C}$ \emph{satisfies $\mathrm{OST}$} if for every $G\in\Grp(\mathbf{C})$ and $X\in\mathbf{C}(G)$, the quotient $G\backslash X$ exists, and moreover, for every section of the quotient map, $\mathrm{OST}$ is satisfied in the above sense.
\end{defn}

Note that if a quotient map has no sections, then it does not affect whether $\mathbf{C}$ satisfies $\mathrm{OST}$. In addition, by considering the ``opposite group" $G^\mathrm{op}$, we see that if $\mathbf{C}$ satisfies $\mathrm{OST}$ for left actions, then it also satisfies $\mathrm{OST}$ for right actions.

Observe that the diagram $$G\times s^*G_X\rightrightarrows G\times (G\backslash X)\xrightarrow{a\circ s} X$$ given above in fact lives in $\mathbf{C}(G)$, where $G$ acts on itself for the first two objects. Thus, if $G\times(-)$ preserves coequalisers and this diagram is a coequaliser in $\mathbf{C}$, then it is automatically a coequaliser in $\mathbf{C}(G)$.

\begin{lemma}\label{ostkpair}
The parallel pair $G\times s^*G_X\rightrightarrows G\times (G\backslash X)$ of Definition \ref{defost} is the kernel pair of $G\times (G\backslash X)\xrightarrow{a\circ s} X$.
\begin{proof}
This can be checked directly, but once again it's much easier to use the same strategy as Lemma \ref{selffree} and consider the Yoneda embedding $y$. As such, the statement only has to be checked for $\mathbf{C}=\Set$, which is easy. One possible concern is that $y$ does not in general preserve coequalisers, but the proof for $\mathbf{C}=\Set$ only relies on $yG\times yX\rightrightarrows yX\to y(G\backslash X)$ being coequalising, not that it is a coequaliser.
\end{proof}
\end{lemma}

\begin{corr}\label{corrost}
The data of Definition \ref{defost} satisfy $\mathrm{OST}$ if and only if $G\times (G\backslash X)\xrightarrow{a\circ s} X$ is regular epic.
\end{corr}

Of course, we could have just used this to define $\mathrm{OST}$ in the first place. The reason we did not is that the parallel maps in Definition \ref{defost} are in the form we will use later and the diagram there also more closely resembles the classical Orbit-Stabiliser Theorem, as we will now see.

\begin{exmp}\label{egost}
\begin{enumerate}[label=(\roman*)]
\item\label{egost1} Consider $\mathbf{C}=\Set$, which has finite limits and coequalisers, and the product functors $C\times(-)$ preserve all colimits (since $\Set$ is Cartesian closed). Regular epimorphisms in $\Set$ are precisely surjections and the map $G\times (G\backslash X)\xrightarrow{a\circ s} X$ is always surjective, so $\Set$ satisfies $\mathrm{OST}$. Observe that the $\mathrm{OST}$ diagram being a coequaliser precisely says that $$X=\coprod_{x\in G\backslash X}G/G_{sx},$$ which expresses both the classical Orbit-Stabiliser Theorem and the decomposition of $X$ into its orbits.

In fact, assuming the axiom of choice (which we always do), the map $G\times (G\backslash X)\xrightarrow{a\circ s} X$ has a section, which is precisely a map $t\colon X\to G$ such that $t(x)sq(x)=x$ for all $x\in X$. That is, we are choosing, for each $x$, an element $t(x)\in G$ that maps $sq(x)$ to $x$, which lie in the same orbit.

\item\label{egost2} Consider $\mathbf{C}=\Pro$, the category of profinite spaces, which also has finite limits and coequalisers, and the product functors $C\times(-)$ preserve coequalisers. The last claim follows from the general fact that for any small category $\mathbf{D}$ with finite limits, inverse limits are exact in $\Pro(\mathbf{D})$ and the inclusion $\mathbf{D}\inj\Pro(\mathbf{D})$ preserves finite limits and all colimits.

Regular epimorphisms are still (continuous) surjections and, whenever a quotient map has a section $s$, the map $G\times (G\backslash X)\xrightarrow{a\circ s} X$ is surjective. Thus, $\Pro$ satisfies $\mathrm{OST}$, even though a quotient map need not have (continuous) sections. Nevertheless, if the action is free, then the quotient map always has a section. It is for this reason that free actions play a crucial role in profinite Bass--Serre theory.
\end{enumerate}
\end{exmp}

\begin{lemma}\label{freeost}
In the context of Definition \ref{defost}, assume that the action is free. Then, the data satisfy $\mathrm{OST}$ if and only if the map $G\times (G\backslash X)\xrightarrow{a\circ s} X$ is an isomorphism.
\begin{proof}
It is easy to see that for a free action, the two parallel maps of Definition \ref{defost} coincide.
\end{proof}
\end{lemma}

\subsection{Action Barr-exact Categories}
\label{sec2-3}

Having defined $\mathrm{OST}$, the next obvious step is to find some examples. The next proposition (Proposition \ref{propexactost}) will show that Barr-exact categories satisfy $\mathrm{OST}$, which gives us many examples of categories that satisfy $\mathrm{OST}$. However, Barr-exactness is in general too strong an assumption for us, since $\Pro$ satisfies $\mathrm{OST}$ but is not Barr-exact (see \cite[Example 5.6]{chaus}). As such, in Definition \ref{defweakexact}, we will define \emph{action Barr-exact categories}, a weakening of Barr-exactness which includes $\Pro$ and gives a nice sufficient condition for a category to satisfy $\mathrm{OST}$.

The following lemma justifies the definition we are about to make.

\begin{lemma}\label{lemmacong}
Suppose $\mathbf{C}$ is regular and $G\in\Grp(\mathbf{C})$ acts on $X\in\mathbf{C}$. Then, the (regular epi)-mono factorisation of the map $G\times X\xrightarrow{(a,\mathrm{pr}_X)} X \times X$ into $$G\times X\overset{e}{\sur} R\overset{m}{\inj} X \times X$$ establishes $R$ as a congruence on $X$.
\begin{proof}
Reflexivity is witnessed by the map $$X=*\times X\xrightarrow{e_G} G\times X\sur R.$$

Consider the map $f\colon G\times X\xrightarrow{(i_G,a)} G\times X$. Symmetry is then witnessed by the dotted map
\[
\begin{tikzcd}
G\times X \arrow[r, "e", two heads] \arrow[d, "e\circ f"'] & R \arrow[ld, dashed] \arrow[d, "\tau\circ m", hook] \\
R \arrow[r, "m"', hook]                                    & X\times X                                          
\end{tikzcd}
\]
where $\tau\colon X\times X\to X\times X$ swaps the two factors.

Finally, consider the pullback $(G\times X)\times_X(G\times X)$, where the first (resp.\ second) factor maps to $X$ via $\mathrm{pr}_X$ (resp.\ $a$), and similarly define the pullback $R\times_XR$. There is a canonical map $(G\times X)\times_X(G\times X)\to R\times_XR$, which is regular epic as $\mathbf{C}$ is regular. Define the map $f'$ to be the composition $$(G\times X)\times_X(G\times X)\to G\times G\times X\xrightarrow{m_G\times\mathrm{id}_X}G\times X\sur R,$$ where the first map ignores the first copy of $X$. Transitivity is then witnessed by the dotted map
\[\begin{tikzcd}
(G\times X)\times_X(G\times X) \arrow[d, "f'"'] \arrow[r, two heads] & R\times_X R \arrow[d] \arrow[ld, dashed] \\
R \arrow[r, "m"', hook]                                        & X\times X                               
\end{tikzcd}\]
\end{proof}
\end{lemma}

\begin{defn}\label{defweakexact}
We say that a regular category $\mathbf{C}$ is \emph{action Barr-exact} if every congruence that arises from group actions is effective. That is, whenever $G\in\Grp(\mathbf{C})$ acts on $X\in\mathbf{C}$, the congruence $R\rightrightarrows X$ from Lemma \ref{lemmacong} is effective, i.e.\ has a coequaliser and is the kernel pair of its coequaliser.
\end{defn}

It is, of course, obvious that Barr-exact categories are action Barr-exact. In addition, by considering opposite groups, we see that if $\mathbf{C}$ is action Barr-exact for left actions, then it is also action Barr-exact for right actions.

\begin{prop}\label{propexactost}
If $\mathbf{C}$ is action Barr-exact, then it satisfies $\mathrm{OST}$.
\begin{proof}
Suppose $G\in\Grp(\mathbf{C})$ acts on $X\in\mathbf{C}$. Let's factorise the map $(a,\mathrm{pr}_X)$ as in Lemma \ref{lemmacong} to obtain a congruence $R\rightrightarrows X$, which is effective by assumption. In particular, it has a coequaliser $G\backslash X$, which must also be the coequaliser of $G\times X\rightrightarrows X$ since $G\times X\sur R$ is epic. Moreover, $R$ is the kernel pair of the quotient map $q\colon X\to G\backslash X$ by assumption. We can assume that $q$ has a section $s$ (or else we are done). In view of Corollary \ref{corrost}, we would like to show that $G\times (G\backslash X)\xrightarrow{a\circ s} X$ is regular epic.

Note that the two maps $X\overset{\mathrm{id}_X}{\underset{s\circ q}{\rightrightarrows}}X$ are coequalised by $q$, so there is an induced map $f\colon X\to R$. Consider the pullback
\[\begin{tikzcd}
P \arrow[d] \arrow[r, two heads,"p"] & X \arrow[d, "f"] \\
G\times X \arrow[r, two heads]   & R               
\end{tikzcd}\]
where $p$ is regular epic because $\mathbf{C}$ is regular. We claim that $p\colon P\to X$ can be identified with $a\circ s$, which will finish the proof. Indeed, giving a map $Y \to P$ is the same as giving a map $Y\xrightarrow{(g,h,k)} (G\times X)\times X$ such that $h=sqk$ and $k=a(g,h)$. Eliminating $k$, this is the same as giving a map $Y\xrightarrow{(g,h)} G\times X$ such that $h=sqa(g,h)$, but since $G\times X\rightrightarrows X\to G\backslash X$ is a coequaliser, the last equation reduces to $h=sqh$. It only remains to note that the diagram $$G\backslash X\xrightarrow{s}X\overset{\mathrm{id}_X}{\underset{s\circ q}{\rightrightarrows}}X$$ is an equaliser.
\end{proof}
\end{prop}

\begin{exmp}\label{eglocallyweakly}
\begin{enumerate}[label=(\roman*)]
\item Let's think about what the proof of Proposition \ref{propexactost} means for $\mathbf{C}=\Set$. In this case, we have that $R=\{(gx,x)\colon g\in G, x\in X\}\subseteq X\times X$ and the map $f\colon X\to R$ sends $x$ to $(x,sq(x))$, which does live in $R$ because we can find some $g\in G$ that maps $sq(x)$ to $x$, which lie in the same orbit. The pullback $P$ is then the set $P=\{(g,x)\colon x=sq(x)\}$ and the map $p\colon P\to X$ sends $(g,x)$ to $gx$. This clearly agrees with the map $a\circ s\colon G\times(G\backslash X)\to X$, which sends $(g,y)$ to $gs(y)$.

In fact, an alternative and quicker (but somewhat indirect) proof of Proposition \ref{propexactost} is to first observe that it holds for $\mathbf{C}=\Set$, and then use the Barr-embedding theorem. We will use this method in Lemma \ref{stablemma}.
\item The category $\Pro$ is in fact action Barr-exact, which is easy to verify directly, but not Barr-exact. On the other hand, regularity alone is not sufficient to imply $\mathrm{OST}$.

Indeed, consider the category $\mathbf{TfAb}$ of torsion-free abelian groups, which is regular (see \cite[Counterexample 2.6.12]{handbook2}). Group objects in $\mathbf{TfAb}$ are just torsion-free abelian groups by Eckmann--Hilton, and it can be checked that the map $$\Z\times \Z\to\Z\colon (m,n)\mapsto 2m+n$$ defines an action of the first copy of $\Z$ on the second in $\mathbf{TfAb}$. The quotient of this action is the trivial group $0$, so the quotient map has a (unique) section $s\colon 0\to\Z$, but the map $$\Z\times 0\xrightarrow{a\circ s}\Z\colon m\mapsto 2m$$ is not regular epic, since regular epimorphisms in $\mathbf{TfAb}$ are precisely surjections.

We have thus shown that the implications Barr-exact $\Rightarrow$ Action Barr-exact $\Rightarrow$ Regular are not reversible.
\item\label{eglocallyweakly3} We say that a category $\mathbf{C}$ is \emph{locally action Barr-exact} if every overcategory $\mathbf{C}_{/Y}$ is action Barr-exact. Every Barr-exact category is locally Barr-exact, hence locally action Barr-exact. It is also clear that every locally action Barr-exact category with a terminal object is action Barr-exact.

It turns out that $\Pro$ is actually also locally action Barr-exact. To see this, let $G\in\Grp(\Pro_{/Y})$ act on $X\in\Pro_{/Y}$ and write $G_y$ (resp.\ $X_y$) for the fibre of $G$ (resp.\ $X$) above $y\in Y$. If we can show that the fibrewise topological quotient $X/{\sim}$ is profinite, where $x,x'\in X_y$ are related if there exists $g\in G_y$ such that $gx=x'$, then the claim would follow easily. Now, note that the equivalence relation ${\sim}\subseteq X\times X$, being the image of $$G\times_YX\xrightarrow{(a,\mathrm{pr}_X)}X\times X,$$ is closed, so $X/{\sim}$ is compact Hausdorff. Total disconnectedness follows from the fact that $Y$ and all fibres $G_y\backslash X_y$ are totally disconnected.
\end{enumerate}
\end{exmp}



\section{Graphs of Groups and Their Fundamental Groups}
\label{sec3}

Our eventual goal is to prove, under suitable assumptions, a categorical Bass--Serre structure theorem for groups acting on trees with tree quotients (Theorem \ref{mainthm}). To do so, we first need to define (internal) fundamental groups. In Sections \ref{sec3-1} and \ref{sec3-2}, we will study $\mathbf{C}$-internal colimits in $\Grp(\mathbf{C})$ and use them to define fundamental groups of graphs of groups. In Theorem \ref{thmregepic} of Section \ref{sec3-3}, we will prove the surjectivity part of the structure theorem.

\subsection{Graphs of Groups}
\label{sec3-1}

Let us first recall the classical structure theorem for $\mathbf{C}=\Set$. Let $G$ be a group acting on a tree $X$. In general, the quotient map $q\colon X\to G\backslash X$ has a section $s$ in $\Set$, but not necessarily in $\Graph=\Graph(\Set)$. Nevertheless, if we assume that the quotient $G\backslash X$ is also a tree, then the quotient map $q$ in fact has a section $s$ in $\Graph$. As in the paragraph before Definition \ref{defost}, we obtain a group object $s^*G_X\in\Grp(\Graph_{/(G\backslash X)})$. In fact, $s^*G_X$ is precisely the \emph{graph of groups associated to the action}. Before we explain why this is equivalent to the classical notion of graphs of groups, let us make the following definition.

\begin{defn}\label{defbary}
Let $X\in\Graph=\Graph(\Set)$. The \emph{barycentric subdivision\footnote{This is defined and denoted by $\int X$ in \cite[Definition 2.2.1]{sdecomp}. If $X$ is the underlying graph of a category, then our $SX^\mathrm{op}$ is precisely the \emph{subdivision category} defined in \cite[Section IX.5]{maccat}.} of $X$} is the category $SX$ with object set $X=VX\coprod EX$, and where there are exactly two non-identity morphisms $e\to d_0(e)$, $e\to d_1(e)$ for every edge $e\in EX$. Composition is defined trivially, since two non-identity morphisms can never be composed. A \emph{graph of groups over $X$} is a functor $SX\to\Grp$.\footnote{It is often assumed in the literature that $\G$ has to map morphisms of $SX$ to \emph{injective} group homomorphisms, but we do not make this assumption.}
\end{defn}

Now, for any $Y\in\Graph$, there is an equivalence of categories $[SY,\Set]=\Graph_{/Y}$ (see Proposition \ref{baryeqprop}), which is really just a special case of the familiar equivalence $$\left[\int Y,\Set\right]=[(\boldsymbol{\cdot}\rightrightarrows\boldsymbol{\cdot}),\Set]_{/Y}.$$ Thus, a graph of groups over $G\backslash X$ is equivalently an object of $\Grp(\Graph_{/(G\backslash X)})$. As a functor $S(G\backslash X)\to\Grp$, the graph of groups $s^*G_X$ we obtained earlier sends $x\in G\backslash X$ to the stabiliser group $G_{sx}$, and sends a morphism $e\to d_i(e)$ to the inclusion of stabilisers $G_{se}\inj G_{sd_i(e)}=G_{d_is(e)}$, which is clearly just the classical notion.

For the classical structure theorem, we next define the \emph{fundamental group of $s^*G_X\colon S(G\backslash X)\to\Grp$} to be $\pi_1(s^*G_X)=\colim s^*G_X$ (cf.\ Example \ref{imptex}) and show that the canonical map of groups $\pi_1(s^*G_X)\to G$ is an isomorphism.

Let's return to a general category $\mathbf{C}$ with finite limits and let $Q\in\Graph(\mathbf{C})$.

Consider the diagonal functor $$\Delta\colon\mathbf{C}\to\Graph(\mathbf{C})_{/Q},\,\,\,\, C\mapsto (C\times EQ\rightrightarrows C\times VQ).$$ This functor preserves finite products, so induces a functor on group objects which is also denoted by $\Delta$: $$\Delta\colon\Grp(\mathbf{C})\to\Grp(\Graph(\mathbf{C})_{/Q}).$$

\begin{defn}\label{fundgrp}
A \emph{graph of groups over $Q$} is an object of $\Grp(\Graph(\mathbf{C})_{/Q})$.

Assume that the functor $$\Delta\colon\Grp(\mathbf{C})\to\Grp(\Graph(\mathbf{C})_{/Q})$$ defined above has a left adjoint, which we denote by $\colim$. Then, we define the \emph{fundamental group of a graph of groups $\G\in\Grp(\Graph(\mathbf{C})_{/Q})$} to be $\pi_1(\G)=\colim \G\in\Grp(\mathbf{C})$.
\end{defn}

The most important examples of graphs of groups are those that arise from group actions. Let $G\in\Grp(\mathbf{C})$, $X\in\Graph(G)$ and suppose the quotient map $q$ has a fixed section $s$ in $\Graph(\mathbf{C})$. As before, we obtain a group object $s^*G_X\in\Grp(\Graph(\mathbf{C})_{/(G\backslash X)})$, the \emph{graph of groups associated to the action}. If the diagonal functor $$\Delta\colon\Grp(\mathbf{C})\to\Grp(\Graph(\mathbf{C})_{/(G\backslash X)})$$ has a left adjoint $\colim$, then we can define the fundamental group $\pi_1(s^*G_X)=\colim s^*G_X\in\Grp(\mathbf{C})$.

We have now made enough definitions to attempt to find conditions for the canonical map $\pi_1(s^*G_X)\to G$ to be an isomorphism. However, to better motivate why we have called the above left adjoint ``$\colim$" and to formulate an internal version of the equivalence $$[S(G\backslash X),\Grp]=\Grp(\Graph_{/(G\backslash X)}),$$ we take a brief detour.

\subsection{Internal Colimits}
\label{sec3-2}

In this subsection, let $\mathbf{C}$ be a category with finite limits. Often, $\mathbf{C}$ will be a lextensive category. By \cite[Propositions 4.5 and 4.8]{extensive}, $\mathbf{C}$ is then in particular locally distributive, i.e.\ $A\times_D(B\coprod C)=(A\times_D B)\coprod (A\times_D C)$. The important takeaway from this part is Proposition \ref{constructcolim2}.

\begin{defn}
Let $\mathbf{C}$ be a lextensive category. Given a graph $X\in\Graph(\mathbf{C})$, the \emph{barycentric subdivision of $X$} is the $\mathbf{C}$-internal category $SX$ defined as follows.
\begin{itemize}
\item It has object of objects $(SX)_0=VX\coprod EX\in\mathbf{C}$.
\item It has object of morphisms $(SX)_1=VX\coprod EX\coprod EX\coprod EX\in\mathbf{C}$, where the first $VX\coprod EX$ represents the identity maps.
\item The middle copy of $EX$ in $(SX)_1$ has source map $EX\inj VX\coprod EX$ and target map $EX\xrightarrow{d_0} VX\inj VX\coprod EX$.
\item The final copy of $EX$ in $(SX)_1$ has source map $EX\inj VX\coprod EX$ and target map $EX\xrightarrow{d_1} VX\inj VX\coprod EX$.
\item Since $\mathbf{C}$ is extensive, the following pullback is initial for $i=0,1$:
\[
\begin{tikzcd}
\varnothing \arrow[r] \arrow[d] & EX \arrow[d, "d_i"] \\
EX \arrow[r, hook]              & VX\coprod EX       
\end{tikzcd}
\]
Consequently, by local distributivity, there is a unique way to define composition in $SX$ that satisfies the axioms of internal categories. Informally, two morphisms in $SX$ cannot be composed unless one of them is an identity map.
\end{itemize}
\end{defn}

Let $[SX,\mathbf{C}]_{\mathbf{C}}$ denote the (ordinary) category of $\mathbf{C}$-internal functors from $SX$ to $\mathbf{C}$.

\begin{prop}\label{baryeqprop}
Let $\mathbf{C}$ be a lextensive category. There is an equivalence $$[SX,\mathbf{C}]_{\mathbf{C}}=\Graph(\mathbf{C})_{/X}.$$
\begin{proof}
Let $P\in[SX,\mathbf{C}]_{\mathbf{C}}$ with total object $P_0\in\mathbf{C}$ and structure maps $$p_0\colon P_0\to (SX)_0,\,\,\,\, p_1\colon (SX)_1\times_{(SX)_0}P_0\to P_0.$$ By extensivity, we can write $p_0$ as the coproduct of maps $VP_0\to VX$ and $Ep_0\colon EP_0\to EX$. Then, by lextensivity and the axioms of internal functors, defining the map $p_1$ is the same as defining maps $f_i\colon EX\times_{(SX)_0} EP_0=EP_0\to P_0$ for $i=0,1$ (corresponding to the middle and final copies of $EX$ in $(SX)_1$), such that the following diagram commutes:
\[\begin{tikzcd}
EP_0 \arrow[r, "f_i"] \arrow[d, "Ep_0"'] & P_0 \arrow[d, "p_0"] \\
EX \arrow[r, "d_i"']                     & (SX)_0              
\end{tikzcd}\]
But this implies that $f_i$ factors through $VP_0$ and the commutative diagram exactly says that $f_0, f_1\colon EP_0\rightrightarrows VP_0$ is an object in $\Graph(\mathbf{C})_{/X}$.

The above defines a functor $[SX,\mathbf{C}]_{\mathbf{C}}\to\Graph(\mathbf{C})_{/X}$. Its inverse is easy to construct and left as an exercise.
\end{proof}
\end{prop}

\begin{rmk}
If $X\in\Graph(\mathbf{C})$ has $EX=\varnothing$, then the above proposition is just the standard equivalence $[C^\delta,\mathbf{C}]_{\mathbf{C}}=\mathbf{C}_{/C}$, where $C^\delta$ denotes the discrete internal category associated to $C\in\mathbf{C}$ (see \cite[Proposition 8.2.5]{handbook}).
\end{rmk}

Under the equivalence in Proposition \ref{baryeqprop}, the diagonal functor $$\Delta\colon\mathbf{C}\to\Graph(\mathbf{C})_{/(G\backslash X)},\,\,\,\, C\mapsto (C\times(G\backslash EX)\rightrightarrows C\times(G\backslash VX))$$ given before becomes $$\Delta\colon\mathbf{C}\to [S(G\backslash X),\mathbf{C}]_{\mathbf{C}},\,\,\,\, C\mapsto C\times (S(G\backslash X))_0=(C\times (G\backslash EX))\coprod(C\times (G\backslash VX)).$$

In fact, for any $\mathbf{C}$-internal category $A$, there is such a diagonal functor $\Delta_A\colon \mathbf{C}\to[A,\mathbf{C}]_{\mathbf{C}}$, which has a left adjoint if $\mathbf{C}$ has finite limits and coequalisers (see \cite[Proposition 8.3.2]{handbook}). Given $P\in[A,\mathbf{C}]_{\mathbf{C}}$, it is standard to call the image of $P$ under the left adjoint the \emph{internal colimit of $P$} and so reasonable to denote it by $\colim P$.

Similarly, there is an induced functor $\Grp(\mathbf{C})\to\Grp([A,\mathbf{C}]_{\mathbf{C}})$, which is also denoted by $\Delta_A$. If $P\in\Grp([A,\mathbf{C}]_{\mathbf{C}})$ has a reflection along $\Delta_A$, we also call the reflection the \emph{(internal) colimit of $P$} and denote it by $\colim P$. If $\Delta_A$ has a left adjoint, we say that \emph{$\Grp(\mathbf{C})$ has all colimits of shape $A$}. If there is such a left adjoint for every $A$, we say that \emph{$\Grp(\mathbf{C})$ has all colimits internal to $\mathbf{C}$}. We have the following familiar fact, which only requires $\mathbf{C}$ to have finite limits.

\begin{prop}\label{constructcolim}
If $\Grp(\mathbf{C})$ has (ordinary) coequalisers and (internal) colimits of shapes $A_0^\delta$ and $A_1^\delta$, then it has (internal) colimits of shape $A$.
\begin{proof}
Given $P\in\Grp([A,\mathbf{C}]_{\mathbf{C}})$, we can construct $\colim P=\colim_AP$ as the coequaliser of two appropriate maps from $\colim_{A_1^\delta}(A_1\times_{A_0}P_0)$ to $\colim_{A_0^\delta}P_0$, exactly as in \cite[Theorem V.2.1]{maccat} or \cite[Theorem 2.3]{jc4}.
\end{proof}
\end{prop}

When $\mathbf{C}$ has finite limits but is not necessarily extensive, we might not be able to make the identification $[SX,\mathbf{C}]_{\mathbf{C}}=\Graph(\mathbf{C})_{/X}$ for a given $X\in\Graph(\mathbf{C})$. Nevertheless, we can still construct $\colim$ using coproducts and coequalisers.

\begin{prop}\label{constructcolim2}
If $\Grp(\mathbf{C})$ has (ordinary) coequalisers and (internal) colimits of shapes $VX^\delta$ and $EX^\delta$, then the diagonal functor $\Delta\colon\Grp(\mathbf{C})\to\Grp(\Graph(\mathbf{C})_{/X})$ has a left adjoint $\colim$.
\begin{proof}
Write $\G\in\Grp(\Graph(\mathbf{C})_{/X})$ as
\[\begin{tikzcd}
E\G \arrow[d] \arrow[r,shift left]\arrow[r,shift right] & V\G \arrow[d] \\
EX \arrow[r,shift left]\arrow[r,shift right]             & VX            
\end{tikzcd}\]
We can then construct $\colim\G$ as the coequaliser of two appropriate maps from $\colim_{EX^\delta}E\G$ to $\colim_{VX^\delta}V\G$. More precisely, the two maps are the adjoints of $E\G\to(\colim_{VX^\delta}V\G)\times EX$, where the first component is the composition $$E\G\xrightarrow{d_i} V\G\to(\colim_{VX^\delta}V\G)\times VX\to\colim_{VX^\delta}V\G$$ for $i=0,1$.
\end{proof}
\end{prop}

This detour has hopefully justified the notation $\pi_1(s^*G_X)=\colim s^*G_X\in\Grp(\mathbf{C})$ from earlier.

\subsection{The Regular Epimorphism $\pi\to G$}
\label{sec3-3}

The goal of this subsection is to prove the ``surjectivity" part of the structure theorem, namely that the canonical map $\pi=\pi_1(s^*G_X)\to G$ is regular epic under suitable assumptions. In this subsection, let $\mathbf{C}$ be a category with finite limits and coequalisers which are preserved by the functors $C\times(-)$, and assume that $\mathbf{C}$ satisfies $\mathrm{OST}$ (see Definition \ref{defost}).

Let $G\in\Grp(\mathbf{C})$, $X\in\Graph(G)$ and suppose the quotient map $q$ has a fixed section $s$ in $\Graph(\mathbf{C})$. By assumption, $s$ satisfies $\mathrm{OST}$, in the sense that $$G\times s^*G_X\rightrightarrows G\times (G\backslash X)\xrightarrow{a\circ s} X$$ is a coequaliser in $\Graph(\mathbf{C})$. Note that there is an abuse of notation here just like in Definition \ref{stabdef}.

Assume that the diagonal functor $$\Delta\colon\Grp(\mathbf{C})\to\Grp(\Graph(\mathbf{C})_{/(G\backslash X)})$$ has a left adjoint $\colim$. Let $\pi=\pi_1(s^*G_X)=\colim s^*G_X\in\Grp(\mathbf{C})$ be the fundamental group as given in Definition \ref{fundgrp}.

There is a canonical map $\varphi\colon\pi\to G$, which is the adjoint of the composition $$s^*G_X\to G_X\to G \times X\to G\times (G\backslash X) =\Delta(G).$$ Note that this composition is related to the second map of the parallel pair in Definition \ref{defost}.

We have made a lot of separate assumptions! For the convenience of the reader, let us summarise them here. We have assumed that:
\begin{itemize}
\item $\mathbf{C}$ is a category with finite limits and coequalisers which are preserved by the functors $C\times(-)$.
\item $\mathbf{C}$ satisfies $\mathrm{OST}$.
\item $G\in\Grp(\mathbf{C})$ and $X\in\Graph(G)$, such that the quotient map has a fixed section $s$ in $\Graph(\mathbf{C})$.
\item The diagonal functor $\Delta$ given above has a left adjoint, so that we can define the fundamental group $\pi$ of the action.
\end{itemize}

\begin{thm}\label{thmregepic}
With the assumptions above, if $X$ is connected, then the map $\varphi\colon\pi\to G$ is regular epic in $\mathbf{C}$.
\begin{proof}
The group object $\pi$ acts naturally on $G$ through $\varphi$ (on the right say), so we can consider the quotient object $G/\pi$, i.e.\ the following coequaliser in $\mathbf{C}$: $$G\times\pi\rightrightarrows G\xrightarrow{\mathrm{coeq}}G/\pi.$$ Here, the two parallel maps are the projection map onto $G$ and the composition $$G\times\pi\xrightarrow{\mathrm{id}_G\times\varphi} G\times G\xrightarrow{m_G}G.$$ Applying $\Delta\colon\mathbf{C}\to\Graph(\mathbf{C})_{/(G\backslash X)}$ and then composing with the map $$\mathrm{id}_{\Delta(G)}\times(s^*G_X\to\Delta(\pi)=\pi\times(G\backslash X)),$$ we obtain the following coequalising diagram in $\Graph(\mathbf{C})_{/(G\backslash X)}$: $$\Delta(G)\times s^*G_X\rightrightarrows \Delta(G)\to\Delta(G/\pi).$$ Its underlying diagram in $\Graph(\mathbf{C})$ is $$G\times s^*G_X\rightrightarrows G\times(G\backslash X)\to\Delta(G/\pi),$$ where the parallel pair is exactly the pair in the $\mathrm{OST}$ diagram for $s$. We therefore obtain a map $X\to\Delta(G/\pi)$ in $\Graph(\mathbf{C})_{/(G\backslash X)}$ and, by Lemma \ref{diaglemma}, its adjoint $\mathrm{Comp}(X)=*\to G/\pi$.

Next, recall that $\Delta$ preserves not only limits, but also coequalisers, by Lemma \ref{deltacoeq}. Consequently, the map $\Delta(G)\to\Delta(G/\pi)$ is regular epic, so the map $X\to\Delta(G/\pi)$ is epic. By Lemmas \ref{deltaff} and \ref{qconnected}, we conclude that $\mathrm{Comp}(X)=*\to G/\pi$ is also epic, so $G/\pi=*$.

Finally, note that $e_G\colon *\to G$ is a section of the quotient map $G\to G/\pi=*$. The composition $\pi=(G/\pi)\times \pi\xrightarrow{a\circ e_G}G$, which is regular epic by the $\mathrm{OST}$ assumption, is just the map $\varphi$.
\end{proof}
\end{thm}

\section{The Standard Graph and the Structure Theorem}
\label{sec4}

In this section, we establish the categorical Bass--Serre structure theorem (Theorem \ref{mainthm}), which gives sufficient conditions for the canonical map $\varphi\colon\pi\to G$ to be an isomorphism, and also a partial converse to the structure theorem for locally action Barr-exact categories (Proposition \ref{propgraphconv}). We define the standard graph $T$ in Section \ref{sec4-1} and give a sufficient condition for $T$ to be connected in Section \ref{sec4-2}. We will never actually define what it means for a connected graph $X$ to be a tree. Instead, the statement ``the canonical map $T\to X$ has a section in $\Graph(\mathbf{C})$" in Theorem \ref{mainthm} should be viewed as an algebraic way of saying that $X$ is a tree (cf.\ Example \ref{mainthmeg}).

\subsection{The Standard Graph}
\label{sec4-1}

We continue with the following general assumptions:
\begin{itemize}
\item $\mathbf{C}$ is a category with finite limits and coequalisers which are preserved by the functors $C\times(-)$.
\item $Q\in\Graph(\mathbf{C})$.
\item The diagonal functor $\Delta\colon\Grp(\mathbf{C})\to\Grp(\Graph(\mathbf{C})_{/Q})$ has a left adjoint $\colim$, so that we can define the fundamental group $\pi=\pi_1(\G)=\colim\G$ of a graph of groups $\G\in\Grp(\Graph(\mathbf{C})_{/Q})$.
\end{itemize}

Of course, the natural example is when $Q=G\backslash X$ for a $G$-graph $X$ with graph section $s$ and $\G=s^*G_X$.

\begin{defn}\label{stdgraph}
Given a graph of groups $\G\in\Grp(\Graph(\mathbf{C})_{/Q})$, its \emph{standard graph $T=T(\G)$} is the following coequaliser in $\Graph(\mathbf{C})$ (and hence in $\Graph(\pi)$): $$\pi\times \G\rightrightarrows \pi\times Q\xrightarrow{\mathrm{coeq}} T.$$ Here, the two parallel maps are given by $\mathrm{id}_\pi\times (\G\to Q)$ and by $$\pi\times \G\to \pi\times\Delta(\pi)=\pi\times\pi\times Q\xrightarrow{m_\pi} \pi\times Q.$$
\end{defn}

\begin{rmk}\label{rmktransport}
We remind the reader that here we are only dealing with the case where there are no transport elements, i.e.\ $Q$ is a ``tree", even though we have not defined trees! Recall that in the last section, we assumed that the quotient map $X\to G\backslash X$ has a \emph{graph} section. If $\mathbf{C}=\Set$, then this happens precisely when there are no transport elements. Similarly, for $\mathbf{C}=\Set$, our definitions of the fundamental group and the standard graph only agree with the classical notions when $Q$ is a tree.
\end{rmk}

\begin{rmk}\label{rmkspecialset}
It is a special property of $\mathbf{C}=\Set$ that if a graph of groups $\G\colon SQ\to\Grp$ sends morphisms of $SQ$ to injective homomorphisms, where $Q$ is a tree, then the vertex/edge groups $\G(x)$, $x\in Q$ embed into $\pi$, or equivalently the map $\G\to\Delta(\pi)$ is monic. In fact, the analogous claim already fails for $\mathbf{C}=\Pro$ (see \cite[Section 6.4]{ribesgraph}). However, this problem disappears if the graph of groups $\G$ comes from the action of a group object $G$ on a graph $X$. In this case, the composition $\G=s^*G_X\to\Delta(\pi)\to\Delta(G)$ is monic, so the first map is also monic.
\end{rmk}

\begin{lemma}\label{samequo}
We have an isomorphism $\pi\backslash T=Q$ in $\Graph(\mathbf{C})$.
\begin{proof}
Recall from Lemma \ref{selftrans} that the canonical action of a group object on itself is transitive. We then obtain the claim from the following diagram of coequalisers, where the bottom left parallel maps clearly coincide.
\[\begin{tikzcd}
\pi\times\pi\times\G \arrow[r,shift left]\arrow[r,shift right] \arrow[d,shift left]\arrow[d,shift right] & \pi\times\pi\times Q \arrow[r] \arrow[d,shift left]\arrow[d,shift right] & \pi\times T \arrow[d,shift left]\arrow[d,shift right] \\
\pi\times \G \arrow[r,shift left]\arrow[r,shift right] \arrow[d]          & \pi\times Q\arrow[d]       \arrow[r]              & T \arrow[d]           \\
\G \arrow[r,shift left]\arrow[r,shift right]                              & Q\arrow[r]                               & \pi\backslash T      
\end{tikzcd}\]
\end{proof}
\end{lemma}

\begin{lemma}\label{sectionpsi}
The map $$t\colon\pi\backslash T=Q=*\times Q\xrightarrow{e_\pi} \pi\times Q\to T$$ is a section of the quotient map $T\to\pi\backslash T$.
\begin{proof}
This follows from the bottom right square of the diagram in the proof of Lemma \ref{samequo}.
\end{proof}
\end{lemma}

\begin{prop}\label{propgraphconv}
Suppose that $\mathbf{C}$ is also locally action Barr-exact and that the map $\G\to\Delta(\pi)$ is monic in $\Graph(\mathbf{C})_{/Q}$. Then, we have a canonical isomorphism $\G=t^*\pi_T$.
\begin{proof}
In fact, we only need the overcategories $\mathbf{C}_{/VQ}$ and $\mathbf{C}_{/EQ}$ to be action Barr-exact. The first observation is that the parallel maps $$\pi\times \G\rightrightarrows \pi\times Q\xrightarrow{\mathrm{coeq}} T$$ which define the standard graph $T$ can be viewed as the action-projection maps of the group object $\G\in\Grp(\Graph(\mathbf{C})_{/Q})$ acting on $\Delta(\pi)=\pi\times Q$ in $\Graph(\mathbf{C})_{/Q}$. The claim then follows by applying the next lemma to $$V\G\inj\Delta(\pi)=\pi\times VQ$$ in the category $\mathbf{D}=\mathbf{C}_{/VQ}$ and the analogous monomorphism in $\mathbf{C}_{/EQ}$.
\end{proof}
\end{prop}

\begin{lemma}\label{stablemma}
Suppose $\mathbf{D}$ is action Barr-exact. Let $f\colon L\inj H$ be a monic map in $\Grp(\mathbf{D})$. Let $u$ be the composition $$*\xrightarrow{e_H}H\to H/L.$$ Then, the canonical map $L\to u^*H_{H/L}$ is an isomorphism, where $H_{H/L}$ is the stabiliser of $H$ acting on the quotient $H/L$.
\begin{proof}
First, observe that this is true for $\mathbf{D}=\Set$. Indeed, the stabiliser of $e_HL\in H/L$ in $H$ is precisely $e_HLe_H^{-1}=L$.

In general, let's first show that there is even an action of $H$ on $H/L$. Note that the map $H\times L\xrightarrow{(\mathrm{pr}_H,a)}H\times H$ is monic, since $f$ is monic (in $\Grp(\mathbf{D})$ and hence in $\mathbf{D}$). By action Barr-exactness, we deduce that $H\times L$ is the kernel pair of $H\to H/L$. In particular, as $\mathbf{D}$ is regular, the functor $H\times(-)$ preserves the coequaliser $$H\times L\rightrightarrows H\to H/L,$$ so the action of $H$ on itself induces an action on $H/L$.

Next, let's construct the canonical map $L\to u^*H_{H/L}$. It's easy to see that $$L\overset{f}{\underset{e_H}{\rightrightarrows}}H\to H/L$$ is coequalising, so there is an induced map
\[\begin{tikzcd}
L \arrow[rd] \arrow[d, "g"', dashed] &                       &     \\
H_{H/L} \arrow[r, "\mathrm{eq}"']    & H\times H/L \arrow[r, shift left]\arrow[r, shift right] & H/L
\end{tikzcd}\]
and therefore an induced map
\[\begin{tikzcd}
L \arrow[rdd, bend right] \arrow[rrd, "g", bend left] \arrow[rd, dashed,"h"] &                                &                   \\
                                                                         & u^*H_{H/L} \arrow[d] \arrow[r] & H_{H/L} \arrow[d] \\
                                                                         & * \arrow[r, "u"']              & H/L              
\end{tikzcd}\]

We now use the Barr-embedding theorem, which says that every $\emph{small}$ regular category can be embedded regularly into a presheaf category $[I,\Set]$. If $\mathbf{D}$ is not small, we can consider the smallest full subcategory $\mathbf{D}_0\subseteq\mathbf{D}$ (which will be small and regular) that contains all the objects in question and that is closed under finite limits and finite colimits taken in $\mathbf{D}$. This embedding preserves finite limits and, importantly, the quotient $H/L$ because it preserves regular epimorphisms and kernel pairs. The result then follows because the image of $h$ in $[I,\Set]$ is an isomorphism.
\end{proof}
\end{lemma}

\subsection{Connectedness}
\label{sec4-2}

We would like to know when the standard graph $T$ is connected. In classical Bass--Serre theory (i.e.\ when $\mathbf{C}=\Set$), this follows from the fact that $\pi$ is, by definition, \emph{generated} by the $\G(x)$, $x\in Q$. We follow the same general assumptions as the last subsection and additionally assume that $Q\in\Graph(\mathbf{C})$ is connected.

\begin{defn}\label{defgen}
Let $G\in\Grp(\mathbf{C})$ and $U\to G$ be a map in $\mathbf{C}$. We say that \emph{$U\to G$ generates $G$}, or simply that \emph{$U$ generates $G$}, if the following diagram is a coequaliser in $\mathbf{C}$: $$G\times U\rightrightarrows G\to*.$$ Here, the two parallel maps are the projection map onto $G$ and the composition $$G\times U\to G\times G \xrightarrow{m_G}G.$$
\end{defn}

\begin{lemma}\label{genlemma}Suppose that $G,H\in\Grp(\mathbf{C})$, that there are two maps $V\to U$ and $U\to G$ in $\mathbf{C}$, and that there is a map $G\to H$ in $\Grp(\mathbf{C})$.
\begin{enumerate}[label=(\roman*)]
\item\label{genlemma1} If $V\to U\to G$ generates $G$, then $U\to G$ generates $G$. The converse holds if $V\to U$ is regular epic.
\item\label{genlemma2} If $U\to G$ generates $G$ and $G\to H$ is epic as a map in $\mathbf{C}$, then $U\to G\to H$ generates $H$.
\end{enumerate}
\begin{proof}
Consider the diagram of coequalisers
\[\begin{tikzcd}
G\times V \arrow[d] \arrow[r,shift left]\arrow[r,shift right] & G \arrow[d,equal] \arrow[r] & C \arrow[d] \\
G\times U \arrow[r,shift left]\arrow[r,shift right]           & G \arrow[r]           & C'          
\end{tikzcd}\]
from which the first part of \ref{genlemma1} follows. For the second part of \ref{genlemma1}, note that the left vertical map is regular epic and hence epic, so the right vertical map is an isomorphism. Part \ref{genlemma2} follows similarly from the following diagram.
\[\begin{tikzcd}
G\times U \arrow[d] \arrow[r,shift left]\arrow[r,shift right] & G \arrow[d] \arrow[r] & C \arrow[d] \\
H\times U \arrow[r,shift left]\arrow[r,shift right]           & H \arrow[r]           & C'          
\end{tikzcd}\]
\end{proof}
\end{lemma}

Recall that we have two diagonal functors $$\mathbf{C}\to\Graph(\mathbf{C})_{/Q}\,\,\,\,\text{and}\,\,\,\,\Grp(\mathbf{C})\to\Grp(\Graph(\mathbf{C})_{/Q}),$$ which are both denoted by $\Delta$, with their respective left adjoints $\mathrm{Comp}$ and $\colim$. In particular, there is a canonical map $\mathrm{Comp}(\G)\to\pi$ in $\mathbf{C}$.

\begin{prop}\label{stdconn}
If the map $\mathrm{Comp}(\G)\to\pi$ generates $\pi$, then the standard graph $T$ is connected.
\begin{proof}
Applying $\mathrm{Comp}$ to the coequaliser diagram in Definition \ref{stdgraph} that defines $T$ and using the fact that $Q$ is connected, we obtain the coequaliser diagram $$\pi\times\mathrm{Comp}(\G)\rightrightarrows\pi\to\mathrm{Comp}(T).$$ It only remains to observe that the parallel pair here coincides with the pair that defines the generation of $\pi$ from $\mathrm{Comp}(\G)$.
\end{proof}
\end{prop}

\begin{rmk}\label{stdrmk}
\begin{enumerate}[label=(\roman*)]
\item\label{stdrmk1} For $\G\in\Grp(\Graph(\mathbf{C})_{/Q})$, let $V\G$ denote its vertex object, i.e.\ the part of $\G$ lying above $VQ$. By Lemma \ref{genlemma}\ref{genlemma1}, we have that $\mathrm{Comp}(\G)$ generates $\pi$ if and only if $V\G$ generates $\pi$. Informally, this is asking for $\pi$ to be generated by the vertex groups $\G(x)$, $x\in VQ$.

In fact, by Proposition \ref{constructcolim2} and Lemma \ref{genlemma}\ref{genlemma2}, if $\Grp(\mathbf{C})$ has the relevant colimits and the forgetful functor $\Grp(\mathbf{C})\to\mathbf{C}$ preserves epimorphisms (or regular epimorphisms, which holds if $\mathbf{C}$ is regular say), then $V\G$ generates $\pi$ provided that $V\G$ generates the internal coproduct $\colim_{VQ^\delta}V\G\in\Grp(\mathbf{C})$. By internal coproducts, we just mean internal colimits in $\Grp(\mathbf{C})$ of discrete shapes $C^\delta$.
\item Let's think about what generation means when $\mathbf{C}=\Set$. Note that if $U\inj G$, then the coequaliser of $G\times U\rightrightarrows G$ is $G/{\sim}$, where ${\sim}$ is the smallest equivalence relation that identifies $g$ with $gu$ for $g\in G$ and $u\in U$. Of course, the coequaliser is $*$ if and only if every element of $G$ can be written as a finite word in $U^{\pm1}$, if and only if $U$ generates $G$ in the usual sense.
\item Consider $\mathbf{C}=\Pro$. If $U$ is a closed subset of a profinite group $G$, then it is still true that $U$ generates $G$ as in Definition \ref{defgen} if and only if $U$ (topologically) generates $G$ in the usual sense. Indeed, for the if direction, we have that $\langle U \rangle\subseteq G$ is dense, so a continuous map $G\to Z$ which is constant on $\langle U \rangle$ must be constant on $G$. Conversely, the quotient map $G\to G/\overline{\langle U \rangle}$ factors through $*$, so $G=\overline{\langle U \rangle}$.
\item For both $\mathbf{C}=\Set$ and $\mathbf{C}=\Pro$, it is true that $V\G$ generates $\pi$ (or even the free product $\colim_{VQ^\delta}V\G\in\Grp(\mathbf{C})$), which follows from the constructions of their respective free products. For the profinite case, see \cite[Proposition 5.1.6]{ribesgraph}.
\item For the case where no transport elements are involved, Proposition \ref{stdconn} gives an alternate, more categorical proof of (a special case of) \cite[Proposition 6.3.4]{ribesgraph}.
\end{enumerate}
\end{rmk}

\subsection{The Structure Theorem}
\label{sec4-3}

We now assume the setup of Theorem \ref{thmregepic}. Let $G\in\Grp(\mathbf{C})$ and $X\in\Graph(G)$ be connected such that the quotient map $X\to G\backslash X$ has a section $s$ in $\Graph(\mathbf{C})$. We can then form the standard graph $T=T(s^*G_X)$, which is connected if $s^*G_{VX}$ generates $\pi=\colim s^*G_X$, by Remark \ref{stdrmk}\ref{stdrmk1}.

Let $\varphi\colon\pi\to G$ be the canonical map, which is regular epic as a map in $\mathbf{C}$ by Theorem \ref{thmregepic}. There is also a map $\Phi\colon T\to X$ in $\Graph(\mathbf{C})$ given as follows.
\[\begin{tikzcd}
\pi\times s^*G_X \arrow[r,shift left]\arrow[r,shift right] \arrow[d,"\varphi"] & \pi\times(G\backslash X) \arrow[d,"\varphi"] \arrow[r] & T \arrow[d, dashed,"\Phi"] \\
G\times s^*G_X \arrow[r,shift left]\arrow[r,shift right]             & G\times(G\backslash X) \arrow[r,"a\circ s"]             & X                               
\end{tikzcd}\]
Let $\Ker\varphi$ be the kernel of $\varphi$, i.e.\ the equaliser $$\Ker\varphi\xrightarrow{\mathrm{eq}}\pi\overset{\varphi}{\underset{e_G}{\rightrightarrows}}G.$$

\begin{lemma}\label{quophi}
The canonical action of $\Ker\varphi$ on $T$ has quotient $X$. In fact, the associated quotient map is precisely $\Phi$.
\begin{proof}
Let $K=\Ker\varphi$. Consider the following diagram in $\Graph(\mathbf{C})$:
\[\begin{tikzcd}
K\times\pi\times s^*G_X \arrow[d, shift left] \arrow[d, shift right] \arrow[r, shift left] \arrow[r, shift right]                          & K\times\pi\times(G\backslash X) \arrow[r] \arrow[d, shift left] \arrow[d, shift right]    & K\times T \arrow[d, shift left] \arrow[d, shift right]       \\
\pi\times s^*G_X \arrow[r, shift left] \arrow[r, shift right] \arrow[d, "\varphi"] & \pi\times(G\backslash X) \arrow[d, "\varphi"] \arrow[r] & T \arrow[d, "\Phi", dashed] \\
G\times s^*G_X \arrow[r, shift left] \arrow[r, shift right]                        & G\times(G\backslash X) \arrow[r, "a\circ s"]            & X                          
\end{tikzcd}\]
Here, the top row is formed by applying $K\times(-)$ to the middle row which is a coequaliser by definition, so all three rows are coequalisers by our assumptions on $\mathbf{C}$. The vertical maps between the top two rows are the projection maps (i.e.\ ignore $K$) and the action maps of $K$ on $\pi$ or $T$. It is not difficult to see that all suitable squares in the diagram commute, e.g.\ in the top left square, the two action maps and the two projections $s^*G_X\to G\backslash X$ form a commutative square.

We claim that the left two columns are also coequalisers. If this were the case, then we would conclude that the rightmost column is a coequaliser as well, which would complete the proof. To prove the claim, it suffices to show that $K\times\pi\rightrightarrows\pi\to G$ is a coequaliser (i.e.\ the First Isomorphism Theorem!).

The key observation is that this is simply the $\mathrm{OST}$ diagram for $\pi$ acting on $G$ through $\varphi$. Indeed, we have from the proof of Theorem \ref{thmregepic} that $\varphi$ can be identified with the map $$\pi=*\times\pi=(G/\pi)\times\pi\xrightarrow{a\circ e_G} G,$$ so it only remains to show that the stabiliser $e_G^*\pi_G$ is precisely $K$. Giving a map into the pullback $e_G^*\pi_G$ is the same as giving a map into $G\times\pi$ whose first component is $e_G$ and which equalises $$G\times\pi\overset{m_G\circ\varphi}{\underset{\mathrm{pr}_G}{\rightrightarrows}}G,$$ but this is clearly the same as giving a map into $K$, as required.
\end{proof}
\end{lemma}

For the next theorem, we shall assume the hypotheses before and in Theorem \ref{thmregepic}, as well as the following:
\begin{itemize}
\item The standard graph $T$ is connected (see Proposition \ref{stdconn} for a sufficient condition).
\item The canonical action of $\Ker\varphi$ on $T$ is free.
\end{itemize}

\begin{rmk}
The assumption that $\Ker\varphi$ acts freely on $T$ may seem arbitrary, but we will see below, in Proposition \ref{exactfree}, that this always holds if we additionally assume that $\mathbf{C}$ is locally action Barr-exact.

In fact, we will later define the notion of \emph{Bass--Serre categories} (Definition \ref{bscat}), which implies most of the assumptions made above and is easier to verify.
\end{rmk}

\begin{thm}\label{mainthm}
With the assumptions above, the following are equivalent.
\begin{enumerate}[label=(\roman*)]
\item\label{mainthm1} The quotient map $\Phi\colon T\to X$ has a section in $\Graph(\mathbf{C})$.
\item\label{mainthm2} The quotient map $\Phi\colon T\to X$ is an isomorphism in $\Graph(\mathbf{C})$.
\item\label{mainthm3} The canonical map $\varphi\colon\pi\to G$ is an isomorphism in $\mathbf{C}$ (and hence in $\Grp(\mathbf{C})$).
\end{enumerate}
\begin{proof}
The implications \ref{mainthm3} $\Rightarrow$ \ref{mainthm2} $\Rightarrow$ \ref{mainthm1} are obvious, so we only have to prove \ref{mainthm1} $\Rightarrow$ \ref{mainthm3}. In view of Theorem \ref{thmregepic}, it only remains to show that $\varphi$ is monic. By Lemma \ref{freeost}, we can make the identification $T=\Ker\varphi\times X$. But since $T$ is connected, this implies that $\Ker\varphi=*$. It then follows by Yoneda that $\varphi$ is monic.
\end{proof}
\end{thm}

\begin{prop}\label{exactfree}
If $\mathbf{C}$ is locally action Barr-exact in addition to the assumptions before Theorem \ref{thmregepic}, then the canonical action of $\Ker\varphi$ on $T$ is free. 
\begin{proof}
As in the proof of Proposition \ref{propgraphconv}, we note that $$\pi\times s^*G_X\rightrightarrows \pi\times (G\backslash X)\xrightarrow{\mathrm{coeq}} T$$ can be viewed as the action-projection maps of $s^*G_X\in\Grp(\Graph(\mathbf{C})_{/(G\backslash X)})$ acting on $\Delta(\pi)$ in $\Graph(\mathbf{C})_{/(G\backslash X)}$. The claim then follows by applying the next lemma to $$s^*G_{VX}\to\Delta(\pi)=\pi\times V(G\backslash X)\xrightarrow{\Delta(\varphi)}\Delta(G)=G\times V(G\backslash X)$$ in the category $\mathbf{C}_{/V(G\backslash X)}$ and the analogous sequence in $\mathbf{C}_{/E(G\backslash X)}$.
\end{proof}
\end{prop}

\begin{lemma}
Suppose $\mathbf{D}$ is action Barr-exact. Let $$L\xrightarrow{f}H\xrightarrow{g}G$$ be in $\Grp(\mathbf{D})$, such that the composition $g\circ f$ is monic. Let $K=\Ker g$. Then, the canonical action of $K$ on the quotient $H/L\in\mathbf{D}$ is free.
\begin{proof}
This is very similar to the proof of Lemma \ref{stablemma}, so we shall be brief. First, observe that this is true for $\mathbf{D}=\Set$. Indeed, an element of $H/L$ can be represented by $hL$, whose stabiliser in $H$ is $hLh^{-1}$. But if $hf(l)h^{-1}\in K$, then $gf(l)=e_G$, so $l=e_L$.

In general, we deduce by action Barr-exactness that $H\times L$ is the kernel pair of $H\to H/L$ and then use the Barr-embedding theorem.
\end{proof}
\end{lemma}

\section{Bass--Serre Categories: Definition and Examples}
\label{sec5}

Now that we have proven a categorical structure theorem of Bass--Serre theory (Theorem \ref{mainthm}), we should investigate when the assumptions of the theorem are met. Let us list the assumptions below and think about how common each one is.
\begin{itemize}
\item $\mathbf{C}$ has finite limits and coequalisers which are preserved by the functors $C\times(-)$.
\end{itemize}
These are common conditions with many well-studied examples and closure properties. For example, if $\mathbf{C}$ is Cartesian closed, then $C\times(-)$ preserves all colimits, so in particular coequalisers. As another example, if $\mathbf{C}$ is a small category with the property in the bullet point, then $\Pro(\mathbf{C})$ has the same property too (cf.\ Example \ref{egost}\ref{egost2}).
\begin{itemize}
\item $\mathbf{C}$ satisfies $\mathrm{OST}$.
\end{itemize}
By Proposition \ref{propexactost}, this is true in all action Barr-exact categories, which include Barr-exact categories and $\Pro$. Alternatively, one can hope that in a nice category, regular epimorphisms are precisely ``surjections" and it is easy to check if $G\times(G\backslash X)\to X$ is surjective.
\begin{itemize}
\item The action of $\Ker\varphi$ on the standard graph $T$ is free.
\end{itemize}
By Proposition \ref{exactfree}, this is true if we assume that $\mathbf{C}$ is locally action Barr-exact (in addition to some of the other assumptions), which again include Barr-exact categories and $\Pro$ (see Example \ref{eglocallyweakly}\ref{eglocallyweakly3}).
\begin{itemize}
\item The diagonal functor $\Delta\colon\Grp(\mathbf{C})\to\Grp(\Graph(\mathbf{C})_{/(G\backslash X)})$ has a left adjoint $\colim$, and the standard graph $T$ is connected.
\end{itemize}
As we saw in Remark \ref{stdrmk}, this holds for $\mathbf{C}=\Set$ and $\mathbf{C}=\Pro$. In general, however, this seems to very much depend on whether $\Grp(\mathbf{C})$ has internal coproducts (see Proposition \ref{constructcolim2}) and how they are constructed.
\begin{itemize}
\item The quotient map $X\to G\backslash X$ has a section in $\Graph(\mathbf{C})$.
\end{itemize}
This property is a major obstacle in profinite Bass--Serre theory (see \cite[Section 6.6]{ribesgraph}). Indeed, \cite[Theorem 6.6.1]{ribesgraph} has to assume that $G\backslash X$ is finite to guarantee a section\footnote{A so-called \emph{fundamental 0-section} in the language of \cite{ribesgraph}, which is not in general a \emph{graph} section. But we remind the reader again that we are only dealing with the case of tree quotients in this paper.} of $X\to G\backslash X$.

This condition is different in flavour from the ones above it, in that it should be viewed as a condition on the specific group and graph objects involved, rather than just on the ambient category $\mathbf{C}$.

\begin{defn}\label{bscat}
We say that a category $\mathbf{C}$ is \emph{Bass--Serre} if it satisfies the following.
\begin{itemize}
\item $\mathbf{C}$ is locally action Barr-exact with a terminal object, as well as coequalisers which are preserved by the functors $C\times(-)$.
\item $\Grp(\mathbf{C})$ has coequalisers and $\mathbf{C}$-internal coproducts (see Proposition \ref{constructcolim2}). Moreover, for any $C\in\mathbf{C}$ and any $H\in\Grp(\mathbf{C}_{/C})$, we have that $H$ generates $\colim H=\colim_{C^\delta} H$, in the sense that $$(\colim H)\times H\rightrightarrows\colim H\to*$$ is a coequaliser (see Definition \ref{defgen} and also Remark \ref{stdrmk}\ref{stdrmk1}).
\end{itemize}
\end{defn}

The following is simply a restatement of Theorem \ref{mainthm}.

\begin{thm}\label{mainthmre}
Let $\mathbf{C}$ be a Bass--Serre category. Suppose $G\in\Grp(\mathbf{C})$ acts on a connected $X\in\Graph(\mathbf{C})$ such that the quotient map $X\to G\backslash X$ has a section in $\Graph(\mathbf{C})$. Then, the following are equivalent.
\begin{enumerate}[label=(\roman*)]
\item\label{mainthmre1} The quotient map $\Phi\colon T\to X$ has a section in $\Graph(\mathbf{C})$.
\item\label{mainthmre2} The quotient map $\Phi\colon T\to X$ is an isomorphism in $\Graph(\mathbf{C})$.
\item\label{mainthmre3} The canonical map $\varphi\colon\pi\to G$ is an isomorphism in $\mathbf{C}$ (and hence in $\Grp(\mathbf{C})$).
\end{enumerate}
\end{thm}

The above ``structure theorem" says in particular that a group acting on a connected graph with appropriate sections can be recovered from the associated graph of groups. We also obtain from Proposition \ref{propgraphconv} (together with Remark \ref{rmkspecialset}) the following converse to the structure theorem, which says that a graph of groups can sometimes be recovered from the associated action.

\begin{thm}\label{mainconv}
Let $\mathbf{C}$ be a Bass--Serre category. Suppose $Q\in\Graph(\mathbf{C})$ is a graph and $\G\in\Grp(\Graph(\mathbf{C})_{/Q})$ is a graph of groups such that $\G\to\Delta(\pi)=\pi\times Q$ is monic. Then, we have a canonical isomorphism $\G=t^*\pi_T$.

In particular, if $\G=s^*G_X$ comes from the action of a group object $G$ on a graph $X$, then $\G\to\Delta(\pi)$ is always monic and therefore $s^*G_X=t^*\pi_T$.
\end{thm}

From our earlier discussions, we have the following two important examples of Bass--Serre categories.

\begin{prop}\label{bssetpro}
The categories $\Set$ and $\Pro$ are Bass--Serre.
\end{prop}

\begin{exmp}\label{mainthmeg}
\begin{enumerate}[label=(\roman*)]
\item Let $\mathbf{C}=\Set$. As we pointed out in Remark \ref{rmktransport}, our notions only agree with the classical ones when $G\backslash X$ is a tree, in which case the quotient map $X\to G\backslash X$ does have a graph section, by \cite[Proposition I.2.6]{dicksdun}.

Let us also justify our previous claim that the condition ``$T\to X$ has a graph section" of Theorem \ref{mainthmre} should be viewed as a proxy for saying that $X$ is a tree. Indeed, by \cite[Proposition I.2.6]{dicksdun} again, if $X$ is a tree, then the quotient map $T\to X$ has a graph section. This does not use the fact that $T=X$ or that $T$ is a tree. On the other hand, if $T\to X$ has a graph section, then $T=X$ from our theorem, so $X$ is a tree because $T$ is.

Importantly, this does not result in a cyclic argument if we simply wish to apply the structure theorem to deduce the structure of a group. Suppose $G$ acts on a tree $X$, such that the quotient $G\backslash X$ is also a tree. Then, we have from above that $T\to X$ has a graph section \emph{without} using the fact that $T=X$ or that $T$ is a tree. Theorem \ref{mainthmre} allows us to deduce that $X=T$ and $G=\pi$.

\item Let $\mathbf{C}=\Pro$. Our notions only agree with the usual ones when $G\backslash X$ is simply connected\footnote{This is different from saying that $G\backslash X$ is a profinite tree: see \cite[Section 3.10]{ribesgraph}.} as a profinite graph, for example when it is a finite tree. In the case that $G\backslash X$ is a finite tree, the quotient map $X\to G\backslash X$ does have a graph section.

Like before, if $X$ is simply connected, then the quotient map $T\to X$ has a graph section, because $T$ is a Galois covering of $X$ (see \cite[Section 3.1]{ribesgraph}). This does not use the fact that $T=X$ or that $T$ is simply connected. The converse, however, relies on these facts (see \cite[Theorem 6.3.5]{ribesgraph}). Nevertheless, if a profinite group $G$ acts on a simply connected profinite graph $X$, such that the quotient $G\backslash X$ is a finite tree, then we can still use Theorem \ref{mainthmre} to deduce that $X=T$ and $G=\pi$.
\end{enumerate}
\end{exmp}

Here are more examples of common categories which are Bass--Serre.

\begin{prop}\label{bspresheaf}
For any small category $I$, the presheaf category $[I,\Set]$ is Bass--Serre.
\begin{proof}
Since $[I,\Set]$ is Barr-exact with coequalisers and Cartesian closed, the first bullet point of Definition \ref{bscat} follows. It is also clear that $\Grp([I,\Set])=[I,\Grp]$ has coequalisers, so the first non-trivial point is the existence of internal coproducts. It's easy to check directly that the diagonal functor $\Delta\colon[I,\Grp]\to \Grp([I,\Set]_{/Y})$ has a left adjoint $\colim$ which can be computed pointwise, for $f\colon F\to Y\in\Grp([I,\Set]_{/Y})$, via $$\colim F(i)=*_{y\in Y(i)}F(i,y),$$ where $*$ denotes the free product of groups and $F(i,y)=f(i)^{-1}(y)$. The generation axiom is equivalent to checking that for each $i\in I$, the set $F(i)=\coprod_{y\in Y(i)}F(i,y)$ generates the group $\colim F(i)=*_{y\in Y(i)}F(i,y)$ in the usual sense, which is obvious.
\end{proof}
\end{prop}

\begin{rmk}
In the proof of Proposition \ref{bspresheaf}, we could have also used the identification $[I,\Set]_{/Y}=[\int Y,\Set]$. Then, the diagonal functor $\Delta\colon[I,\Grp]\to [\int Y,\Grp]$ is just composition with $p\colon\int Y\to I$, so its left adjoint $\colim$ is simply the left Kan extension $\mathrm{Lan}_p$, which exists because $\Grp$ is cocomplete. However, knowing its existence is not sufficient for the proof. We needed to know how it is constructed to check the generation axiom.
\end{rmk}

\begin{prop}\label{bssheaf}
For any small Grothendieck site $I$, the Grothendieck topos $\mathrm{Sh}(I,\Set)$ is Bass--Serre.
\begin{proof}
The first bullet point of Definition \ref{bscat} is once again obvious, and $\Grp(\mathrm{Sh}(I,\Set))=\mathrm{Sh}(I,\Grp)$ has coequalisers which are computed as the sheafifications of the presheaf coequalisers. Internal coproducts of group objects can similarly be computed as the sheafifications of the formula in the proof of Proposition \ref{bspresheaf}. That is, given some $f\colon F\to Y\in\Grp(\mathrm{Sh}(I,\Set)_{/Y})$, we define $\colim F\in\mathrm{Sh}(I,\Grp)$ to be the sheafification of the presheaf $$i\mapsto *_{y\in Y(i)}F(i,y).$$ The generation axiom follows from properties of sheafification.
\end{proof}
\end{prop}

\begin{rmk}
The last two propositions clearly give us many examples of Bass--Serre categories. However, to apply Theorem \ref{mainthmre}, we still need the quotient maps $X\to G\backslash X$ and $T\to X$ to have graph sections, which in the context of the last two propositions are natural transformations. Even if we know that a presheaf $G\backslash X\in\Graph([I,\Set])=[I,\Graph]$ is pointwise a tree, there is no guarantee that a \emph{natural} section $G\backslash X\to X$ exists. Thus, from our perspective, a major obstacle in generalising Bass--Serre theory, just like in the profinite case, is the lack of sections.
\end{rmk}

\appendix
\section{Fundamental Groups as Weighted 2-colimits}
\label{appendixA}

In the main body of the paper, we focused on the case where there are no transport elements, i.e.\ when $G\backslash X$ is a tree (see Remark \ref{rmktransport}). Of course, we know that classical Bass--Serre theory works for any quotient $G\backslash X$. The goal of this appendix is to suggest an idea for how we can extend the paper to include the general case by considering $\mathbf{C}=\Set$, and why it is probably more technical. We will assume basic knowledge of Bass--Serre theory and of 2-category theory. Throughout the appendix, we only work with the base category $\mathbf{C}=\Set$. Also, let $\Cat$ denote the category of small categories.

The reason for making the following definition will be justified in Example \ref{imptex} and Proposition \ref{thm1}, where we show that classical fundamental groups of graphs of groups can be described as weighted (strict) 2-categorical colimits.

\begin{defn}\label{def1}
Let $X$ be a connected graph and $T$ a maximal subtree. Recall from Definition \ref{defbary} that $SX\in\Cat$ denotes the barycentric subdivision of $X$.

We define the \emph{weight associated to $({X},T)$} to be the functor $W=W({X},T)\colon{SX}^\mathrm{op}\to\Cat$ given as follows. On objects, $W$ sends every element of $T$ to the terminal category $*$ and every element of the complement ${X}-T$ (which must be an edge) to the interval category $\mathbf{2}=(0\to 1)$. Given a non-identity morphism $e^i\colon e\to d_i(e)$ in ${SX}$, where $e\notin ET$, the map $W(e^i)\colon*\to\mathbf{2}$ sends the unique object of $*$ to $i\in\mathbf{2}$.
\end{defn}

Let us also remind the reader what the general fundamental group (including transport elements) of a graph of groups $\G\colon SX\to\Grp$ is.

\begin{defn}\label{def2}
Let ${X}$ be a connected graph with maximal subtree $T$ and $\G$ be a graph of groups over ${X}$. The \emph{fundamental group of $\G$ with respect to $T$} is the group $\pi=\pi_1(\G,T)$ (which is independent of $T$ up to isomorphism) given by the following presentation. It has generators $$\left(\bigsqcup_{v\in {VX}}\G(v)\right)\sqcup\{t_e\colon e\in EX\}$$ and relations $$\left(\bigsqcup_{v\in {VX}}\{\text{relations in }\G(v)\}\right)\sqcup\{t_e[\G(e^0)(g)]t_e^{-1}=\G(e^1)(g)\colon e\in EX, g\in \G(e)\}\sqcup\{t_e=1\colon e\in ET\}.$$
\end{defn}

In the following, we view $\Grp$ as a 2-category by embedding it into the 2-category $\Cat$ in the usual way, i.e.\ by viewing each group as a one-object category. Notice that given $G,H\in\Grp$ and two 1-morphisms (i.e.\ group homomorphisms) $\varphi,\psi\colon G\to H$, a 2-morphism $t\colon \varphi\Rightarrow\psi$ is the same as an element $t\in H$ such that $t\varphi(g)t^{-1}=\psi(g)$ for all $g\in G$.

\begin{exmp}\label{imptex}
\begin{enumerate}[label=(\roman*)]
\item Suppose ${X}$ is a star graph. That is, let $Y$ be any set and define $VX=Y\sqcup\{p\}$, $EX=\overline{Y}\cong Y$, $d_0(\overline{y})=p$ and $d_1(\overline{y})=y$. This is a tree, so the only maximal subtree is $T={X}$.

Let $\G$ be the graph of groups which assigns every edge of ${X}$, as well as the star vertex $p$, the trivial group, and which assigns a vertex $y\in Y$ some group $G_y$. The fundamental group of $\G$ (with respect to $T$) is the free product (i.e.\ coproduct in $\Grp$) of the $G_y$, which we denote by $\coprod_YG_y$.

The weight $W$ associated to $T$ is trivial, i.e.\ it sends every object of ${SX}$ to $*\in\Cat$. Thus, if the $W$-weighted colimit of $\G$ exists, then it must just be given by the 1-colimit of $\G$, which is easily verified to be $\coprod_YG_y$. To show that this is in fact the $W$-weighted colimit of $\G$, we still have to check the relevant two-dimensional universal property. That is, we want to show that for each $G\in\Grp$, the bijection of sets $$\Hom_{\Grp}\left(\coprod_YG_y,G\right)\cong\Hom_{[{SX}^\mathrm{op},\Cat]}(W,\Hom_\Grp(\G(-),G))\cong\Hom_{[{SX},\Grp]}(\G,\Delta G)$$ is actually an isomorphism of categories (natural in $G$). We leave this as an exercise, since we will be dealing with a more general case in Proposition \ref{thm1}.

Observe that $\coprod_YG_y$ is the 1-coproduct but not the 2-coproduct of the $G_y$ in $\Grp$. The two-dimensional universal property above relies on the connectedness of ${SX}$.

\item Suppose ${X}$ is a single edge $e$ with $d_0(e)=u$, $d_1(e)=v$, which is again a tree. Note that the barycentric subdivision ${SX}$ is simply a span. Let $\G$ be the graph of groups which assigns groups $G_u, G_v, G_e$ to $u, v, e$ respectively. The fundamental group of $\G$ is the pushout $G_u\coprod_{G_e}G_v$, and so is the $W$-weighted colimit of $\G$.

\item Suppose ${X}$ is a single loop, i.e.\ a single vertex $v$ and a single edge $e$, which has a unique maximal subtree $T=\{v\}$. Let $\G$ be the graph of groups which assigns groups $G,H$ to $v,e$ respectively, and morphisms $\alpha,\beta\colon H\to G$. The fundamental group of $\G$ is the HNN-extension of $\alpha$ and $\beta$.

The barycentric subdivision ${SX}$ is the category $(e\rightrightarrows v)$ and weight $W\colon{SX}^\mathrm{op}\to\Cat$ sends $v$ to $*\in\Cat$, $e$ to $\mathbf{2}\in\Cat$, the top arrow to $*\mapsto0\in\mathbf{2}$ and the bottom arrow to $*\mapsto1\in\mathbf{2}$. The $W$-weighted colimit of $\G$ is, by definition, the 2-coinserter of $\alpha$ and $\beta$, which is precisely their HNN-extension.
\end{enumerate}
\end{exmp}

In the above, we had to check both the one-dimensional and two-dimensional universal properties of the relevant 2-colimit. If the 2-category $\Grp$ has all copowers by $\mathbf{2}$, then the two-dimensional universal property would be automatic (see \cite[page 6]{kelly2cat}). Unfortunately, this is not the case, as the following example shows.

\begin{exmp}
The 2-category $\Grp$ does not have all copowers by $\mathbf{2}$ (and so in particular is not 2-cocomplete). In fact, even the copower $\mathbf{2}*1$ does not exist. If it does, then for each $G\in\Grp$, we would have an isomorphism of categories $$\Hom_{\Grp}(\mathbf{2}*1,G)\cong\Hom_{\Cat}(\mathbf{2},\Hom_\Grp(1,G))\cong\Hom_\Cat(\mathbf{2},G),$$ natural in $G$. Looking at objects, this forces $\mathbf{2}*1=\Z$. Let $g,h\in G\cong\Hom_\Grp(\Z,G)$. Then, a morphism $g\to h$ in $\Hom_\Grp(\Z,G)$ corresponds precisely to an element $t\in G$ conjugating $g$ to $h$, but a morphism $g\to h$ in $\Hom_\Cat(\mathbf{2},G)$ corresponds to any element of $G$.
\end{exmp}

In view of Example \ref{imptex}, the following should not be surprising. We note that when ${X}$ is a tree (so that the weight $W$ is trivial), the one-dimensional aspect of the following proposition was already observed in \cite[Section 3.9]{sdecomp}.

\begin{prop}\label{thm1}
Let ${X}$ be a connected graph with maximal subtree $T$, ${SX}$ be its barycentric subdivision, and $W\colon{SX}^\mathrm{op}\to\Cat$ be the associated weight. Let $\G\colon{SX}\to\Grp$ be a graph of groups over ${X}$. Then, the fundamental group $\pi=\pi_1(\G,T)$ of $\G$ is precisely the $W$-weighted 2-colimit of $\G$.
\begin{proof}
By definition, we want to show that for each $G\in\Grp$, there is an isomorphism of categories
\begin{eqnarray}\label{eqn1}
\Hom_{\Grp}(\pi,G)\cong\Hom_{[{SX}^\mathrm{op},\Cat]}(W,\Hom_\Grp(\G(-),G)).
\end{eqnarray}
This will be clear once we describe explicitly what the objects and morphisms in the category $\mathbf{H}=\Hom_{[{SX}^\mathrm{op},\Cat]}(W,\Hom_\Grp(\G(-),G))$ are.

An object $\alpha\in\mathbf{H}$ is precisely a function that assigns each $v\in VX$ a group homomorphism $\alpha_v\colon \G v\to G$, together with an element $a_e\in G$ for each $e\in EX- ET$, such that for every edge $e\in EX$ with $d_0(e)=u$ and $d_1(e)=v$, the following holds.
\begin{itemize}
\item If $e\in ET$ is a tree edge, then the following diagram commutes
\[
\begin{tikzcd}
\G e \arrow[r, "\G(e^0)"] \arrow[d, "\G(e^1)"'] & \G u \arrow[d, "\alpha_u"] \\
\G v \arrow[r, "\alpha_v"']                     & G                         
\end{tikzcd}
\]
\item If $e\notin ET$ is not a tree edge, then the following is a 2-morphism in $\Grp$
\[
\begin{tikzcd}
\G e \arrow[r, "\G(e^0)"] \arrow[d, "\G(e^1)"'] & \G u \arrow[d, "\alpha_u"] \arrow[ld, "a_e" description, Rightarrow] \\
\G v \arrow[r, "\alpha_v"']                     & G                                                       
\end{tikzcd}
\]
In other words, we have $a_e[(\alpha_u\circ\G(e^0))(g)]a_e^{-1}=(\alpha_v\circ\G(e^1))(g)$ for all $g\in \G e$.

(Alternatively, we can combine the two cases by assigning an element $a_e\in G$ for each $e\in EX$, but requiring that $a_e=1$ whenever $e\in ET$.)
\end{itemize}
It is plain from this description and the presentation of $\pi$ that (\ref{eqn1}) holds on objects.

Morphisms are trickier. Suppose we have two objects $\alpha,\beta\in\mathbf{H}$ (with associated elements $a_e,b_e\in G$ respectively). For convenience, let's write $\alpha_e^i=\alpha_{d_i(e)}\circ \G(e^i)$, so that $a_e\alpha_e^0a_e^{-1}=\alpha_e^1$, and similarly for $\beta$. By definition, a morphism $m\colon\alpha\to\beta$ in $\mathbf{H}$ is a modification between natural transformations. Unwinding the definition carefully, this means that:
\begin{itemize}
\item For each $v\in VX$, we have an element $m_v\in G$ such that $m_v\alpha_vm_v^{-1}=\beta_v$.
\item For each $e\in EX$, we have an element $m_e\in G$ such that $m_e\alpha_e^0m_e^{-1}=\beta_e^0$ and an element $m_e'\in G$ such that $m_e'\alpha_e^1m_e'^{-1}=\beta_e^1$. These elements satisfy $m_e'=b_em_ea_e^{-1}$ (so it is enough to specify only $m_e$).
\item For each $e\in EX$, we have, from the modification equations, that $m_e=m_{d_0(e)}$ and $m_e'=m_{d_1(e)}$. For the convenience of the reader, we explicitly display below the modification equation at a morphism $e^0\colon e\to d_0(e)=u\in{SX}$. In particular, if $e\in ET$ is a tree edge, then $m_e=m_{d_0(e)}=m_{d_1(e)}$.
\[
\begin{tikzcd}
Wu=* \arrow[rr, "\alpha_u", bend left] \arrow[rr, "\beta_u"', bend right] \arrow[dd, "*\mapsto0"'] & m_u\bigg\Downarrow  & {\Hom_\Grp(\G u,G)} \arrow[dd, "(-)\circ\G(e^0)"] \\
                                                                                                   &                     &                                                   \\
We \arrow[rr, "\beta_e^0"']                                                                          &                     & {\Hom_\Grp(\G e,G)}                               \\
                                                                                                   & =                   &                                                   \\
Wu=* \arrow[rr, "\alpha_u"] \arrow[dd, "*\mapsto0"']                                               &                     & {\Hom_\Grp(\G u,G)} \arrow[dd, "(-)\circ\G(e^0)"] \\
                                                                                                   &                     &                                                   \\
We \arrow[rr, "\alpha_e^0", bend left] \arrow[rr, "\beta_e^0"', bend right]                            & m_e\bigg\Downarrow & {\Hom_\Grp(\G e,G)}                              
\end{tikzcd}
\]
\end{itemize}
We conclude that $m_x=m_y$ for all $x,y\in{X}$, so a morphism $m\colon\alpha\to\beta$ is equivalently an element $m\in G$ such that $m\alpha_vm^{-1}=\beta_v$ for all $v\in VX$ and $ma_em^{-1}=b_e$ for all $e\in EX$. It is plain from this description that (\ref{eqn1}) is an isomorphism of categories.

We leave it as an exercise to check the naturality of (\ref{eqn1}) in $G$.
\end{proof}
\end{prop}

In \cite{bsgrpd}, Bass--Serre theory for groupoids is studied. In fact, it can be shown that fundamental groupoids of graphs of groupoids can also be described as weighted 2-colimits, in exactly the same way as Proposition \ref{thm1}. Since this is not the focus of the paper, we will be brief here and leave the proof of the next proposition as an exercise.

Let $\Grpd$ be the category of groupoids (i.e.\ small 1-categories where all morphisms are invertible), which again inherits a 2-categorical structure from $\Cat$. Let $X$ be a connected graph with barycentric subdivision ${SX}$. To match the setting of \cite{bsgrpd}, we only consider graphs of groupoids $\G\colon{SX}\to\Grpd$ where $\G(e^i)\colon\G(e)\to\G(d_i(e))$ is an isomorphism on objects for each $e\in EX$ and $i=0,1$ (see \cite[Hypothesis 1]{bsgrpd}).

\begin{prop}\label{thm2}
Let ${X}$ be a connected graph with maximal subtree $T$, ${SX}$ be its barycentric subdivision, and $W\colon{SX}^\mathrm{op}\to\Cat$ be the associated weight. Let $\G\colon{SX}\to\Grpd$ be a graph of groupoids over ${X}$ satisfying the assumption above. Then, the fundamental groupoid of $\G$ (as defined in \cite[Definition 6.1]{bsgrpd}) is precisely the $W$-weighted 2-colimit of $\G$.
\end{prop}

Propositions \ref{thm1} and \ref{thm2} suggest that one way to extend the main body of the paper to include transport elements is to use a categorical framework that combines both internal and enriched category theories. The theory of enriched indexed categories (\cite{shulmanenriched}) comes to mind, but deciding whether it's a good framework for our purpose is beyond the scope of the current paper.

\section{Relations to Categories of Graphs of Groups in \cite{gogpullbacks}}
\label{appendixB}

In \cite{gogpullbacks}, the authors define various categories of graphs of groups and study pullbacks in them. Let us explain exactly how our categorical definition of graphs of groups relates to one of their categories, namely the one given by \cite[Definition 1.5]{gogpullbacks}.

Recall from Definition \ref{defbary} (and the paragraph below it) that for an (ordinary) graph $X\in\Graph$, we define $[SX,\Grp]=\Grp(\Graph_{/X})$ to be the category of graphs of groups over $X$. To allow the base graph $X$ to vary, we would like to take the Grothendieck construction of $$\Graph^{\mathrm{op}}\to\Cat,\,\,\,\, X\mapsto[SX,\Grp],$$ where we note that $S$ defines a functor $\Graph\to\Cat$. One can check that this does not give the same notion as \cite{gogpullbacks}: the problem is that a morphism in $[SX,\Grp]$ is, by definition, a natural transformation, which does not give the \emph{twisting elements} used in \cite[Definition 1.5]{gogpullbacks}.

Instead, as in Appendix \ref{appendixA}, we will view $\Grp$ as a full subcategory of $\Cat$, so that 2-morphisms in $\Grp$ exactly describe conjugation. Let $[SX,\Grp]_{\mathrm{Ps}}$ denote the category of functors and pseudonatural transformations from $SX$ (viewed as a 2-category with only identity 2-morphisms) to $\Grp$. In our specific setting, because $SX$ has no non-identity 2-morphisms and no composable non-identity 1-morphisms, pseudonatural transformations have the following simple description. Given two functors $\G,\Hc\colon SX\to\Grp$, a pseudonatural transformation $\eta\colon\G\Rightarrow \Hc$ consists of:
\begin{itemize}
\item a function that assigns each $x\in X$ a group homomorphism $\eta_x\colon\G(x)\to\Hc(x)$, and
\item for $i=0,1$, a function that assigns each edge $e\in EX$ an element $s_e^i\in\Hc(d_i(e))$ satisfying $$(s_e^i)(\Hc(e^i)\circ \eta_e)(s_e^i)^{-1}=\eta_{d_i(e)}\circ\G(e^i),$$ where we recall that $e^i\colon e\to d_i(e)$ is the morphism in $SX$ that represents the source (if $i=0$) or target (if $i=1$) map at $e$.
\end{itemize}

Consider now the Grothendieck construction $$\mathbf{G}=\int[S-,\Grp]_{\mathrm{Ps}}.$$ By definition, it has objects pairs $(\G,X)$, where $X\in\Graph$ and $\G\colon SX\to\Grp$, i.e.\ $\G$ is a graph of groups over $X$. Given two such objects $(\G,X)$ and $(\Hc,Y)$, a morphism $\mu\colon(\G,X)\to(\Hc,Y)$ in $\mathbf{G}$ is precisely a map $b\colon X\to Y$ in $\Graph$ and a map $\G\to[Sb,\Grp](\Hc)=\Hc\circ Sb$ in $[SX,\Grp]_{\mathrm{Ps}}$. But this means exactly that we have:
\begin{itemize}
\item a function that assigns each $x\in X$ a group homomorphism $\mu_x\colon\G(x)\to\Hc(b(x))$, and
\item for $i=0,1$, a function that assigns each edge $e\in EX$ an element $s_e^i\in\Hc(bd_i(e))$ satisfying $$(s_e^i)(\Hc(b(e^i))\circ \mu_e)(s_e^i)^{-1}=\mu_{d_i(e)}\circ\G(e^i).$$
\end{itemize}

This precisely recovers \cite[Definition 1.5]{gogpullbacks}, as soon as we only consider graphs of groups $\G$ which send morphisms in $SX$ to \emph{injective} group homomorphisms, and realise that we are working with directed graphs, while \cite{gogpullbacks} works with graphs in the sense of Serre \cite{trees}. For the convenience of the reader, we note that our symbols $$b\colon X\to Y,e\in EX,d_0(e),d_1(e),\G(e^0),\G(e^1),s_e^0,s_e^1$$ that appeared earlier correspond to the respective symbols $$[.]\colon \textsf{B}\to \textsf{A},f\in E(\textsf{B}),v,v',\alpha_f,\omega_f,f_\alpha,f_\omega$$ in \cite[Definition 1.5]{gogpullbacks}.

\bibliographystyle{unsrt}

\end{document}